\documentclass[11pt]{article}

\usepackage{amsmath}
\usepackage{epsfig, amssymb, amsfonts, dsfont}
\usepackage{amsthm}
\usepackage{amstext}
\usepackage{amsopn}
\usepackage{mathrsfs}
\usepackage{subfigure}
\usepackage{graphicx}
\usepackage[left=2.3cm,top=2cm,right=2.3cm,bottom=2cm, nohead,foot=1cm]{geometry}
\usepackage[makeroom]{cancel}
\usepackage{color}
\usepackage{enumerate}
\usepackage{comment}
\usepackage{stackrel}

\def\stackrel#1#2{\mathrel{\mathop{#2}\limits^{#1}}}

\numberwithin{equation}{section}

\newtheorem{theorem}{Theorem}[section]

\newtheorem{lemma}[theorem]{Lemma}

\newtheorem{corollary}[theorem]{Corollary}
\newtheorem{proposition}[theorem]{Proposition}

\theoremstyle{definition}
\newtheorem{definition}[theorem]{Definition}
\newtheorem{remark}[theorem]{Remark}
\newtheorem{example}[theorem]{Example}

\newcommand{\R}{{\mathbb R}}

\DeclareFontFamily{U}{mathx}{\hyphenchar\font45}
\DeclareFontShape{U}{mathx}{m}{n}{
      <5> <6> <7> <8> <9> <10>
      <10.95> <12> <14.4> <17.28> <20.74> <24.88>
      mathx10
      }{}
\DeclareSymbolFont{mathx}{U}{mathx}{m}{n}
\DeclareMathSymbol{\bigtimes}{1}{mathx}{"91}

\usepackage{relsize}

\date{\vspace{-9ex}}

\title{The conditional Gaussian multiplicative chaos   structure underlying a critical  continuum random polymer model on a  diamond fractal}

\date{  }

  \author{ \textbf{Jeremy Thane Clark}\footnote{ {\tt
jeremy@olemiss.edu}} \vspace{.1cm}  \\  University of Mississippi, Department of Mathematics   }

\begin{document}
\maketitle

\begin{abstract}
We discuss a Gaussian multiplicative chaos (GMC) structure underlying a family  of random measures  $\mathbf{M}_r$, indexed by $r\in\R$, on a space $\Gamma$ of directed pathways crossing a diamond fractal with Hausdorff dimension two.  The laws of these random continuum path measures arise in a critical weak-disorder limiting regime  for discrete directed polymers on disordered hierarchical graphs.  For the analogous subcritical continuum polymer model in which the diamond fractal has Hausdorff dimension less than two, the random path measures can be constructed as  subcritical GMCs through  couplings to a spatial Gaussian white noise. This construction fails in the critical dimension two where, formally, an infinite coupling strength to the environmental noise would be required to generate the disorder. We prove, however, that there is a conditional GMC interrelationship between the random measures $(\mathbf{M}_r)_{r\in \R}$ such that the law of $\mathbf{M}_r$ can be constructed as a subcritical GMC with random reference measure $\mathbf{M}_R$ for any choice of $R\in (-\infty, r)$.  A similar GMC structure plausibly would hold for a critical  continuum (2+1)-dimensional directed polymer model.

\end{abstract}

\section{Introduction}

 A Gaussian field    $\mathbf{W}$ on a measure space  $(\Gamma,\mu)$ and defined over a  probability space $ (\Omega,\mathscr{F}, \mathbb{P} )  $ is  a bounded linear map $\mathbf{W}: L^2(\Gamma, \mu ) \rightarrow  L^2(\Omega,
\mathscr{F},\mathbb{P} )$ in which the range of $\mathbf{W}$ is a  Gaussian subspace of $L^2(\Omega, \mathscr{F},\mathbb{P} )$.   
  For $\psi\in L^2(\Gamma,\mu )$ we have the alternative notations
\begin{align}\label{TestFun}
\mathbf{W}(\psi)\,\equiv \, \langle\mathbf{W}, \psi\rangle \,\equiv\, \int_{\Gamma}  \mathbf{W}(p) \psi(p)\mu(dp)\,,   
\end{align}
 where the field ``variables" $\{\mathbf{W}(p)  \}_{p\in \Gamma }$, in nontrivial cases, can only be understood in a distributional sense as defining  random variables  when integrated against an appropriate test function $\psi $.  A \textit{Gaussian multiplicative chaos} (GMC) on $(\Gamma,\mu  )$   generated by the field $\mathbf{W}\equiv\{\mathbf{W}(p)  \}_{p\in \Gamma }$ for a coupling strength $\beta\geq 0$ is a random measure $\mathbf{M}_{\beta}$ on $\Gamma$ that is formally expressed as
\begin{align}\label{FormalGMC}
\mathbf{M}_{\beta}(dp)\,=\, e^{ \beta \mathbf{W}(p)-\frac{\beta^2}{2}\mathbb{E}[ \mathbf{W}^2(p)]}\mu(dp)  \,.
\end{align}
In general,  questions of existence and uniqueness for random measures of this type require an indirect  technical interpretation since the field variables $  \mathbf{W}(p)$ are not actual random variables, and thus the factor $ \textup{exp}\big\{ \beta \mathbf{W}(p)-\frac{\beta^2}{2}\mathbb{E}[ \mathbf{W}^2(p)]\big\}$ is not simply a random Radon-Nikodym derivative, $d\mathbf{M}_{\beta}/d\mu  $. In nontrivial cases, $\mathbf{M}_{\beta}$ is a.s.\ mutually singular to $\mu$ even though the expectation measure $\mathbb{E}[\mathbf{M}_{\beta}]$ is absolutely continuous with respect to $\mu$.

The first rigorous approach to defining GMC measures was introduced by Kahane~\cite{Kahane} in 1985.  Background for understanding the motivation for Kahane's work and  a discussion of some of the important  directions that GMC theory has taken since then can be found in the review~\cite{Rhodes} by Rhodes and Vargas; see also Junnila's dissertation~\cite{Junnila1} for a summary of  GMC theory that includes discussion of some additional recent contributions.  In this article we will use the definitional framework for GMC  proposed by  Shamov in~\cite{Shamov}.

 A GMC $\mathbf{M}_{\beta}$ is said to be \textit{subcritical} if the expectation measure, $\mathbb{E}[\mathbf{M}_{\beta}]$, is  $\sigma$-finite and \textit{critical} otherwise.  In the subcritical case, it can be assumed that $\mathbb{E}[\mathbf{M}_{\beta}]=\mu$ and that the covariance operator $\textbf{T}$ of the field $\mathbf{W}$ has a kernel $T(p,q)=\mathbb{E}\big[\mathbf{W}(p)\mathbf{W}(q)   \big]$  related to the point  correlations of $\mathbf{M}_{\beta}$ as follows:
\begin{align}\label{Joint}
\mathbb{E}\big[\mathbf{M}_{\beta}(dp) \mathbf{M}_{\beta}(dq)  \big]\,=\,e^{\beta^2 T(p,q)}\mu(dp)\mu(dq)\,.  
\end{align}
In particular, the measure $\upsilon_{\beta} :=\mathbb{E}[\mathbf{M}_{\beta}\times\mathbf{M}_{\beta} ]$ on $\Gamma\times \Gamma$ is absolutely continuous with respect to the product measure  $\mu\times \mu=\mathbb{E}[\mathbf{M}_{\beta}]\times \mathbb{E}[\mathbf{M}_{\beta}] $.\footnote{See Lemma 34 and the discussion following it in~\cite{Shamov}.}   The GMC formalism~(\ref{FormalGMC}) potentially defines a family of laws  $(\mathbf{M}_{\beta})_{\beta\geq 0}$ for random measures on the space $\Gamma$ such that for $0\leq \alpha<\beta$ the law of $\mathbf{M}_{\beta}$ can be formally constructed from $\mathbf{M}_{\alpha}$ as
\begin{align}\label{RelativeGMC}
\mathbf{M}_{\beta}(dp)\,\stackrel{\mathcal{L}}{=}\,e^{ \sqrt{\beta^2 -\alpha^2} \mathbf{W}_{\mathsmaller{\mathbf{M}_{\alpha}}}(p)-\frac{\beta^2-\alpha^2}{2}\mathbb{E}[ \mathbf{W}^2_{\mathsmaller{\mathbf{M}_{\alpha}}}(p)\,|\, \mathbf{M}_{\alpha}]}\mathbf{M}_{\alpha}(dp)\,,
\end{align}
where the field $\{\mathbf{W}_{\mathbf{M}_{\alpha}}(p)\}_{p\in \Gamma}$ is Gaussian  with kernel $T(p,q)=\mathbb{E}\big[\mathbf{W}_{\mathbf{M}_{\alpha}}(p) \mathbf{W}_{\mathbf{M}_{\alpha}}(q)\,|\,\mathbf{M}_{\alpha}  \big]$  when conditioned on $\mathbf{M}_{\alpha}$.  This conditional form for the field means that $\mathbf{W}_{\mathbf{M}_{\alpha}}:  L^2(\Gamma,\mathbf{M}_{\alpha} ) \rightarrow  \frak{G}$ is a  bounded linear map depending measurably on $\mathbf{M}_{\alpha}$, where $\frak{G}$ 
is a Gaussian subspace of $ L^2(\Omega,\mathscr{F},\mathbb{P} )$ whose variables are jointly independent of $\mathbf{M}_{\alpha}$.

In this article we will show that a family of  random measures, $( \mathbf{M}_r )_{r\in \R}$, introduced in~\cite{Clark4} satisfies a  constructive  GMC interrelationship analogous to~(\ref{RelativeGMC}) even though the random measures  $ \mathbf{M}_r$ are not GMCs  with respect to a deterministic ``pure" measure $\mu$ as in~(\ref{FormalGMC}).    The  measures $\mathbf{M}_r $ are defined on a space $\Gamma$ of directed pathways through a compact diamond fractal $D$ having Hausdorff dimension two, and their laws derive from a  continuum/weak-disorder limiting regime for models of  directed polymers on disordered hierarchical graphs.
 The family of random measure laws $(\mathbf{M}_r)_{r\in \R}$ satisfies properties (I)-(V) below for any fixed $r,R\in \R$ with $R\leq r$, where proving (IV) \& (V) is the focus of coming sections. 
\begin{enumerate}[(I)]

\item $\mathbb{E}[ \mathbf{M}_r ]=\mu$, where $\mu$ is a probability measure on $\Gamma$.  Moreover, the law of  $\mathbf{M}_r$ converges to the deterministic measure $\mu$  as $r\searrow -\infty$.

\item Unlike~(\ref{Joint}), the correlation measure $\upsilon_r:=\mathbb{E}[ \mathbf{M}_r \times  \mathbf{M}_r]$ on  $\Gamma\times \Gamma$ is not absolutely continuous with respect to the product measure $\mu\times \mu$.

\item   The Radon-Nikodym derivative of $\upsilon_{r}$ with respect to  $\upsilon_{R}$ is $\exp\{ (r-R) T(p,q) \}$ for a nonnegative kernel $T(p,q)$, and the product measure $\mu\times \mu$ is supported on the set of pairs $(p,q)\in \Gamma\times \Gamma$ with $T(p,q)=0$ (the paths $p$ and $q$ are essentially nonintersecting).  In contrast, the random product measure $\mathbf{M}_r \times  \mathbf{M}_r$ a.s.\ assigns positive weight to the set of $(p,q)$ such that $T(p,q)>0$.

\item  For a.e.\ realization of the random measure $(\Gamma,\mathbf{M}_R)$  there is a field $\big\{\mathbf{W}_{\mathbf{M}_R}(p)\big\}_{p\in \Gamma}$ that is Gaussian with correlation kernel $T(p,q)$ when conditioned on $\mathbf{M}_R$  and for which the following  GMC, $\mathds{M}_{R,r-R}$, is well-defined:  
\begin{align*}
 \mathds{M}_{R,r-R}(dp)\,=\, e^{  \sqrt{r-R}\mathbf{W}_{\mathbf{M}_R}(p)-\frac{r-R}{2}\mathbb{E}[ \mathbf{W}_{\mathbf{M}_R}^2(p)\,|\,\mathbf{M}_R ]}\mathbf{M}_R(dp) \,. 
 \end{align*}

\item The random measure $(\Gamma,\mathds{M}_{R,r-R})$ is equal in law to $(\Gamma,\mathbf{M}_{r})$.

\end{enumerate}
Property (I) implies that $\mathbf{M}_r$ is not a critical GMC since $\mathbb{E}[ \mathbf{M}_r ]$ is finite, and (II) implies that $\mathbf{M}_r$ is not a subcritical GMC since it does not satisfy $\mathbb{E}[ \mathbf{M}_r \times  \mathbf{M}_r] \ll \mathbb{E}[ \mathbf{M}_r] \times \mathbb{E}[ \mathbf{M}_r]  $. Intuitively, since $M_{R}$ approaches $\mu$ as $R\rightarrow -\infty$ and  the coupling strength $\beta=\sqrt{r-R}$ in (IV) blows up with $-R\gg 1$,  the measures $\mathbf{M}_r$ are too heavily disordered to admit a subcritical GMC form with respect to $\mu$.  The assertion in (III)  suggests a strong form of locality for the continuum disordered polymer model defined by $\mathbf{M}_r$ in which the ``disordered environment"  effectively confines the polymers to a  measure zero portion of the space in such a way that independently chosen paths can have nontrivially richer intersection sets than is possible under the pure system.

  Analogous  continuum directed polymer models  on disordered diamond fractals with Hausdorff dimension less than two have a  canonical subcritical GMC construction~\cite{Clark2}, so  dimension two is critical in this family of models. As is the case for the continuum (1+1)-dimensional directed polymer model defined by Alberts, Khanin, and Quastel in~\cite{alberts2}, the random measures defining subcritical continuum polymer models on  diamond fractals  can  be constructed through well-behaved Wiener chaos expansions; see~\cite[Theorem 1.25]{Clark2}.    We conjecture that an analogous conditional GMC structure to that outlined in (I)-(V) would hold for a continuum polymer model  emerging from the critical weak-disorder limiting regime studied by Caravenna, Sun, and Zygorus~\cite{CSZ4} for (2+1)-dimensional polymers. This limiting regime corresponds to the critical case of the  2d stochastic heat equation also recently studied  by Gu, Quastel, and Tsai in~\cite{GQT}.  For a discussion highlighting some of the similarities of the critical regime of  the $(2+1)$-dimensional polymer and our hierarchical model,  we refer the reader to the introduction of~\cite{Clark3}.

\section{GMC definition and a statement of the main result}\label{SecGMC}

Before defining the continuum polymer  model considered in this article, we will state Theorem~\ref{ThmMain}, which is a more precise formulation of the interrelational GMC structure of the family $(\mathbf{M}_r)_{r\in\R}$ summarized in points (IV) \& (V) above.  To do this, we first  turn to the definition of subcritical GMC introduced in~\cite{Shamov} and summarize some of the important theorems in that framework (Section~\ref{SecGMCDef}).  We then form a slight generalization of the GMC definition that allows having a random reference measure and coupling to the field (Section~\ref{SecGMCExt}).\vspace{.2cm}   

We will maintain the following notational and terminological conventions:
\begin{itemize}
\item All $L^p$ spaces in this text refer to real-valued functions.
\item  $\mathcal{H}$ always denotes an infinite-dimensional separable real Hilbert space.
\item The triple $(\Omega,\mathscr{F}, \mathbb{P} )$ denotes the underlying probability space, and $\mathbb{E}$ denotes its corresponding expectation.
\item A linear isometry $W:\mathcal{H}\rightarrow L^2(\Omega,\mathscr{F}, \mathbb{P} )$ whose range is a Gaussian linear space is referred to as a \textit{Gaussian field on} $\mathcal{H}$ or, alternatively, a \textit{standard Gaussian random vector in} $\mathcal{H}$.

\end{itemize}

\subsection{A  formulation of  GMC through Cameron-Martin  shifts of the Gaussian field}\label{SecGMCDef}

In the discussion below,  $\mu$ denotes a  finite measure on a standard Borel measurable space $(\Gamma,\mathcal{B}_{\Gamma})$.  The approach in~\cite{Shamov} to addressing questions of existence and uniqueness of a GMC on $(\Gamma,\mu)$ over a Gaussian field $\mathbf{W}\equiv \{\mathbf{W}(p)\}_{p\in \Gamma}$ 
begins by reformulating the field $\mathbf{W}$ through a pair of linear operators $(W,Y)$ on a  Hilbert space  $\mathcal{H}$  in which $W:\mathcal{H}\rightarrow L^2(\Omega,\mathscr{F}, \mathbb{P} )$ is a Gaussian field on $\mathcal{H}$, and the map $Y:\mathcal{H}\rightarrow L^0(\Gamma ,\mu )$ is  continuous, where the codomain $L^0(\Gamma ,\mu )$ denotes the space of $\mu$-equivalence classes of measurable functions on $\Gamma$ endowed with the topology of convergence in measure.  In the terminology of~\cite{Shamov}, the  linear map $Y$ is referred to as a \textit{generalized $\mathcal{H}$-valued function},  and we can view $Y$ as a function that  sends elements $p \in\Gamma$ to generalized vector values  $Y(p)$ having ``inner product"  $\langle Y(p),\phi\rangle :=(Y\phi)(p) $ with $\phi \in \mathcal{H}$  for  $\mu$-a.e.\ $p$.  Similarly, $W$ determines a random generalized vector  $ W(\omega)$  with  $\langle W(\omega),\phi\rangle :=  (W\phi)(\omega)$ for $\mathbb{P}$-a.e.\ $\omega\in \Omega$.  The generalized values $Y(p)$ and $W(\omega)$ can be viewed as elements of a Frechet space  containing $\mathcal{H}$ as a subspace, for instance, by choosing an  orthonormal basis $(\phi_n)_{n\in \mathbb{N}}$ of $\mathcal{H}$ and identifying  $Y(p)\equiv \big(\langle Y(p),\phi_n\rangle\big)_{n\in \mathbb{N}}$ and $W(\omega)\equiv \big(\langle W(\omega),\phi_n\rangle\big)_{n\in \mathbb{N}}$, i.e., as elements of  $\R^{\mathbb{N}} $.  Moreover, a function $Y:\Gamma\rightarrow \R^{\mathbb{N}}$ determines a generalized $\mathcal{H}$-valued function provided that for any $(\alpha_n)_{n\in \mathbb{N}} \in \ell^2$   the sequence of functions $ \sum_{n=1}^N \alpha_{n}   Y_n(p) $ converges in $L^0(\Gamma,\mu)$   as $N\rightarrow \infty$  to a limit that defines $\langle Y(p),\phi\rangle$ for $\phi = \sum_{n=1}^\infty \alpha_n \phi_n$.  The Gaussian field $\mathbf{W}$ is then formally given by $ \mathbf{W}(p)= \big\langle W, Y(p)\big\rangle  $ through
\begin{align}
\int_{\Gamma}\mathbf{W}(p)\psi(p)\mu(dp)\,:=\,\bigg\langle  W, \int_{\Gamma}Y(p)\psi(p)\mu(dp)\bigg\rangle \,
\end{align}
for a suitable class of test functions $\psi \in L^0(\Gamma,\mu)$ such that $\int_{\Gamma}Y(p)\psi(p)\mu(dp)\in \mathcal{H}$.  The covariance operator of $\mathbf{W}$ is  determined by the quadratic form sending $\psi$ to $  \big\|\int_{\Gamma}Y(p)\psi(p)\mu(dp)   \big\|^2_{\mathcal{H}}      $ and has kernel formally expressed as
$$ T(p,q)\,=\,\mathbb{E}\big[ \mathbf{W}(p)\mathbf{W}(q)  \big] \,=\,\big\langle Y(p),\,Y(q)\big\rangle\hspace{.5cm}\text{for}\hspace{.5cm}p,q\in \Gamma\,.  $$
The  equivalence between defining the Gaussian field $ \mathbf{W}$ through the pair $(W,Y)$ and through  integration against $L^2$ test functions as in~(\ref{TestFun}) is explained in~\cite[Appendix A]{Shamov}. 

The following definition   provides an abstract characterization of a  GMC formally satisfying \begin{align}\label{ReGMCFormal}
\mathbf{M}(dp)= e^{\langle W, Y(p)\rangle -\frac{1}{2}\mathbb{E}[\langle W, Y(p)\rangle^2] } \mu(dp)\,,
\end{align}
 where we  specialize the  statement to the relevant subcritical case.  The basic (trivial) case of this GMC form is when $Y$ defines an $\mathcal{H}$-valued function (i.e., $ Y(p)\in \mathcal{H}   $ for $\mu$-a.e.\ $p\in \Gamma $)---as opposed to merely a \textit{generalized} $\mathcal{H}$-valued function taking values in $\R^{\mathbb{N}}$.

\begin{definition}\label{DefGMC} Let $W$ be a  Gaussian field on $\mathcal{H}$ and $Y:\mathcal{H}\rightarrow L^{0}(\Gamma,\mu)$ be linear and continuous.  A   \textit{subcritcal Gaussian multiplicative chaos} $\mathbf{M}$  over   $(W,Y)$ with expectation $\mu$ is a random finite  measure on $\Gamma$ satisfying (I)-(III) below. 

\begin{enumerate}[(I)]

\item  $\mathbb{E}[ \mathbf{M}]= \mu$, i.e., $\mathbb{E}[ \mathbf{M}(A)]= \mu(A)$  for any measurable set $A\subset \Gamma$.

\item $\mathbf{M}$ is measurable with respect to the field $W$ and can thus be expressed as a measurable function of $W$.

\item  For $\phi\in \mathcal{H}$ and  a.e.\ realization of the field $W$, there is the equality between measures
$$ \mathbf{M}(W+\phi, dp)\,=\,e^{ \langle Y(p),\phi\rangle  }\mathbf{M}(W, dp)   \,. $$

\end{enumerate}

\end{definition}
$\mathbf{M}$ is also referred to as the \textit{GMC associated to $Y$ with expectation $\mu$} in contexts where only the law of the GMC  is relevant to the discussion and not its relationship to the underlying field $W$.
\begin{remark}\label{RemarkFun} As a consequence of (I) in Definition~\ref{DefGMC}, a null set  $A\subset \Gamma$ of  $\mu$ is a.s.\ a null set of $\mathbf{M}$.  In particular,  $\mu$-equivalent measurable functions $\psi_1$ and $\psi_2$  on $\Gamma$  are a.s.\ $\mathbf{M}$-equivalent even though $\mathbf{M}$ is a.s.\ mutually singular to $\mu$ in nontrivial cases.  Thus, measures $\psi_1(p)\mathbf{M}(W, dp) $ and  $\psi_2(p)\mathbf{M}(W, dp) $ are a.s.\ equal for  representatives $\psi_1$ and $\psi_2$ in the  $\mu$-equivalence class of $ e^{\langle Y(p),\phi\rangle }$, which shows that the right side of the equality in (III) does not have ambiguity arising from  this consideration. 

\end{remark}


 The formula in property (III) of Definition~\ref{DefGMC} determines how the GMC measure is changed by a Cameron-Martin shift of the underlying Gaussian field; see~\cite[Chapter 14]{Janson} for a discussion of Cameron-Martin shifts.  If $\mathcal{P}$ denotes the probability measure of the Gaussian field $W$---viewed as an $\R^{\mathbb{N}}$-valued random variable---and $\mathcal{P}_{\phi}$ is the push-forward measure of $\mathcal{P}$ by the shift map $S_{\phi}(\varphi )=\varphi+\phi$, then $\mathcal{P}_{\phi}$ is absolutely continuous with respect to  $\mathcal{P}$ and has Radon-Nikodym derivative 
\begin{align} \label{CMS}
 \frac{d\mathcal{P}_{\phi}}{d\mathcal{P}}\,=\,\textup{exp}\Big\{\langle W,\phi\rangle -\frac{1}{2}\|\langle W,\phi\rangle\|^2   \Big\}\,.
 \end{align}
   With this relation in mind, the GMC formalism~(\ref{ReGMCFormal}) suggests the identity
\begin{align}\label{Identity}
\mu(dp)\mathcal{P}_{Y(p)}(dW)\,=\,  \mathbf{M}(W, dp) \mathcal{P}(dW) \,,
\end{align}
where $\mathcal{P}_{Y(p)}$ denotes the probability measure of the shifted field $W+Y(p)$. Note that by integrating out $p\in \Gamma$ in~(\ref{Identity}) we get that the total mass $ \mathbf{M}(W, \Gamma)$ is the Radon-Nikodym derivative of $\int_{\Gamma}\mathcal{P}_{Y(p)}\mu(dp) $ with respect to $ \mathcal{P}$. Based on these observations, the theorem below gives an important criterion  for the existence and uniqueness of a GMC over $(W,Y)$.

\begin{definition}  Let $\mathcal{P}$ denote the measure on $\mathbb{R}^{\mathbb{N}}$ determined by a  Gaussian standard random vector $W$ on  $\mathcal{H}$.  A continuous linear map $Y:\mathcal{H}\rightarrow L^{0}(\Gamma,\mu)$  is said to be a \textit{randomized shift} if the measure  $\widetilde{\mathcal{P}}=\int_{\Gamma}\mathcal{P}_{Y(p)}\mu(dp) $, i.e., the marginal of $Q(dW,dp)=\mu(dp)\mathcal{P}_{Y(p)}(dW)$,  is absolutely continuous with respect to $\mathcal{P}$.

\end{definition}

 In the trivial case that $Y(p)\in \mathcal{H}$ for $\mu$-a.e.\ $p\in \Gamma$, i.e., the operator $Y$ defines an $\mathcal{H}$-valued function, $Y$ is a randomized shift as a consequence of the Cameron-Martin formula~(\ref{CMS}). 
\begin{theorem}[Theorem 17 \& Corollary 18 of~\cite{Shamov}]\label{ThmShamov} There exists a subcritical GMC over $(W,Y)$ with expectation $\mu$   iff $Y$ is a randomized shift.  The GMC is unique in law and also as a function of $W$ when it exists. 
\end{theorem}
The GMC associated to the randomized shift $Y$ can be denoted by $\mathbf{M}_Y$ although we drop the subscript when $Y$ is  implicitly understood. By~\cite[Corollary 20]{Shamov},  if $Y$ is a randomized shift, then the  covariance kernel $\mathbf{T}$ of field $(W,Y)$ must have an integral kernel $T(p,q)\in L^0(\Gamma\times \Gamma, \mu\times \mu)$.  Moreover, by~\cite[Lemma 34]{Shamov}, when $Y$ is a randomized shift and $\mathbf{M}$ is the associated  GMC with expectation $\mu$, 
\begin{align}\label{ExpectProd}
\mathbb{E}\big[ \mathbf{M}(dp)\mathbf{M}(dq)  \big]\,=\,e^{T(p,q)}\mu(dp)\mu(dq)\,
\end{align}
and, in particular, $\mathbb{E}\big[ \mathbf{M}\times \mathbf{M}  \big]$ is absolutely continuous with respect to $\mathbb{E}[ \mathbf{M}]\times \mathbb{E}[\mathbf{M}  \big]$.

The following convergence theorem is important for deriving a GMC  as a limit of GMCs $\mathbf{M}_{Y_n}$ associated to a sequence of randomized shifts $Y_n$ that converge strongly to a limit $Y$. 
\begin{theorem}[Theorem 25 of~\cite{Shamov}]\label{ThmShamovConv}  Let $Y_n$ be a sequence of randomized shifts and $\mathbf{M}_{Y_n}$ be the associated subcritical GMCs with expectation $\mu$. Consider the following statements.
\begin{enumerate}[(I)]
\item  The family of random variables $\big\{\mathbf{M}_{Y_n}(\Gamma)\big\}_{n\in \mathbb{N}}$ is uniformly integrable.   

\item $Y_n$ converges strongly to a generalized $\mathcal{H}$-valued function $Y:\mathcal{H}\rightarrow L^0(\Gamma,\mu)$, i.e.,  $\langle Y_n(p), \phi\rangle$ converges  to $\langle Y(p), \phi\rangle$ in $L^0(\Gamma,\mu)$ for every $\phi\in \mathcal{H}$.

\item  The kernels $T_n(p,q)=\langle Y_n(p), Y_n(q)\rangle $ converge in $L^0(\mu\times \mu)$ to $T(p,q)=\langle Y(p), Y(q)\rangle $.

\end{enumerate}
Statements (I) and (II) imply that $Y$ is a randomized shift, and thus $Y$ defines a subcritical GMC $\mathbf{M}_Y$.    The statements (I)-(III) imply that the sequence of GMCs $\mathbf{M}_{Y_n}$ converges to $\mathbf{M}_{Y}$ as $n\rightarrow \infty$ in the sense that for any $\psi\in L^1(\Gamma,\mu)$
$$ \int_{\Gamma}\psi(p) \mathbf{M}_{Y_n}(dp) \hspace{.4cm}\stackrel{L^1  }{\Longrightarrow}  \hspace{.4cm} \int_{\Gamma}\psi(p) \mathbf{M}_{Y}(dp) \,.  $$

\end{theorem}
Typically Theorem~\ref{ThmShamovConv} would be applied with $Y$ being a nontrivial generalized $\mathcal{H}$-valued function  and the approximating operators $Y_n$ being $\mathcal{H}$-valued functions, which are trivial randomized shifts. The following is a corollary of Theorem~\ref{ThmShamovConv} that  is not explicitly stated in~\cite{Shamov}. 
\begin{corollary}\label{CorApprox} If $Y:\mathcal{H}\rightarrow L^2(\Gamma,\mu)$ is a bounded linear map for which the operator $YY^*$ is Hilbert-Schmidt with integral kernel $T(p,q)$ satisfying $  \int_{\Gamma\times \Gamma}\textup{exp}\big\{T(p,q)\big\}\mu(dp)\mu(dq)\,<\,\infty  $,
then $Y$ is a randomized shift.
\end{corollary}
  The integral $\int_{\Gamma\times \Gamma}\textup{exp}\big\{T(p,q)\big\}\mu(dp)\mu(dq)$ is the second moment of the total mass of a GMC associated to $Y$ with expectation $\mu$.  More generally, we have Kahane's moment formula~(\ref{Formula}) for the higher moments of the random variable $\mathbf{M}(A)$ for a measurable set $A\subset \Gamma$; see~\cite[part (d) of Theorem 6]{Kahane}.
\begin{proposition}\label{PropGMCmoments} Let $Y:\mathcal{H}\rightarrow L^2(\Gamma,\mu)$ be a bounded linear map for which the operator $YY^*$ is Hilbert-Schmidt with integral kernel $T(p,q)$ having finite exponential moments with respect to $\mu\times \mu$. Let $\mathbf{M}$ be the GMC associated to $Y$ with expectation $\mu$.   For any measurable set $A\subset \Gamma$, the positive integer moments of the random variable $\mathbf{M}(A) $ are finite and have the form
\begin{align}\label{Formula}
 \mathbb{E}\left[ \big(\mathbf{M}(A)  \big)^m  \right]\,=\,\int_{A^m}\exp\Bigg\{ \sum_{1\leq i < j \leq m  }T(p_i, p_j)  \Bigg\}\mu(dp_1)\cdots  \mu(dp_m)  \,.  
 \end{align}

\end{proposition}
We include the proofs of Corollary~\ref{CorApprox} and Proposition~\ref{PropGMCmoments} in Appendix~\ref{SecZ}.

\subsection{Extending the GMC definition to a random reference measure}\label{SecGMCExt}

We now generalize Definition~\ref{DefGMC} to the case where  $\mu$ is replaced by a random finite measure $M$ and the coupling $Y\equiv Y_M$ is a function of $M$. As before, $(\Gamma,\mathcal{B}_{\Gamma})$ denotes a standard Borel measurable space.

\begin{definition}\label{DefCGMC}  Let $(\Gamma,M)$ be a random finite  measure  and    $W$ be a  Gaussian standard random vector in $\mathcal{H}$ independent of $M$, in other terms, for which the random variables in $\textup{Range}(W)\subset L^2(\Omega,\mathscr{F}, \mathbb{P} )$ are jointly independent of $M$. Moreover, let $Y_M:\mathcal{H}\rightarrow L^{0}(\Gamma,M)$ be a continuous linear map depending measurably on $M$.  A  \textit{conditional GMC}  over $(W,Y_M)$ with conditional expectation $M$ is a random finite measure $\mathbf{M}$ on $\Gamma$ satisfying (I)-(III) below. 
\begin{enumerate}[(I)]

\item  $\mathbb{E}[ \mathbf{M}\,|\,M]=M $, i.e., $\mathbb{E}[ \mathbf{M}(A)\,|\,M]= M(A)$  for any measurable set $A\subset \Gamma$.

\item $\mathbf{M}$ is measurable with respect to the $\sigma$-algebra generated by the random measure $M$ and the Gaussian field $W$, and thus $\mathbf{M}$ is a function of the pair $(M,W)$.

\item  For $\phi\in \mathcal{H}$ and  a.e.\ realization of the pair $(M,W)$, there is the equality between measures
$$ \mathbf{M}\big(M, W+\phi, dp\big)\,=\,e^{ \langle Y_M(p), \phi \rangle  }\mathbf{M}\big(M,W, dp\big)   \,. $$

\end{enumerate}

\end{definition}
The following is an immediate consequence of Theorem~\ref{ThmShamov}.
\begin{corollary}\label{CorShift} Let  $Y_M:\mathcal{H}\rightarrow L^{0}(\Gamma,M)$ and $W:\mathcal{H}\rightarrow L^2(\Omega,\mathscr{F}, \mathbb{P} )$ be as in Definition~\ref{DefCGMC}.  There is a unique conditional GMC over $(W,Y_M)$ with conditional expectation $M$   if and only if  $Y_M$ is a randomized shift for a.e.\ realization of $M$.
\end{corollary}

As  formally expressed in~(\ref{RelativeGMC}), a basic example of a conditional GMC can be constructed with   $M \stackrel{d}{=}\mathbf{M}_{\alpha } $ and $\textbf{M} \stackrel{d}{=} \mathbf{M}_{\beta } $, where the GMCs $\mathbf{M}_{\alpha }\equiv\mathbf{M}_{\alpha Y} $ and $\mathbf{M}_{\beta }\equiv\mathbf{M}_{\beta Y} $ are associated to randomized shifts $\alpha Y$ and $\beta Y$ with $0\leq \alpha\leq \beta$  and have expectation $\mu$:
\begin{example} Let $Y:\mathcal{H}\rightarrow L^2(\Gamma,\mu)$  be a bounded linear map for which $YY^*$ is Hilbert-Schmidt with kernel $T(p,q)$ satisfying $\int_{\Gamma\times \Gamma}\textup{exp}\big\{\beta T(p,q)   \big\}\mu(dp)\mu(dq)<\infty$. By Corollary~\ref{CorApprox},  $\alpha Y$ is a randomized shift for any  $ \alpha\in [0,\beta]$.
\begin{itemize}

\item  The GMC associated to $\beta Y$ with expectation $\mu$ can be constructed as $$\mathbf{M}_{\beta }\bigg(\Big(1-\frac{\alpha^2}{\beta^2}\Big)^{1/2}W +\frac{\alpha}{\beta}W' ,\,dp\bigg)\,, $$ where $W,W':\mathcal{H}\rightarrow L^2(\Omega,\mathscr{F}, \mathbb{P} )$  are independent Gaussian fields over $\mathcal{H}$.  

\item  The operator $Y\equiv Y_{\mathbf{M}_{\alpha }}$  a.s.\ defines  a bounded operator $Y_{\mathbf{M}_{\alpha }}:\mathcal{H}\rightarrow L^2(\Gamma,\mathbf{M}_{\alpha } )$  for which $(\beta^2-\alpha^2)^{1/2}Y_{\mathbf{M}_{\alpha }}$ is a randomized shift since $\int_{\Gamma\times\Gamma  }\textup{exp}\big\{(\beta^2-\alpha^2)T(p,q)\big\}\mathbf{M}_{\alpha }(dp)\mathbf{M}_{\alpha }(dq)   $ is a.s.\ finite, which we can see from the formula $\mathbb{E}\big[\mathbf{M}_{\alpha }(dp)\mathbf{M}_{\alpha }(dq)    \big]=\textup{exp}\big\{ \alpha^2 T(p,q) \big\}\mu(dp)\mu(dq)$:
$$ \mathbb{E}\bigg[\int_{\Gamma\times\Gamma  }e^{(\beta^2-\alpha^2)T(p,q)}\mathbf{M}_{\alpha }(dp)\mathbf{M}_{\alpha }(dq)      \bigg]\,=\,\int_{\Gamma\times\Gamma  }e^{\beta^2T(p,q)}\mu(dp)\mu(dq) \,<\,\infty  \,. $$

\item $\mathbf{M}_{\beta }\left(\big(1-\frac{\alpha^2}{\beta^2}\big)^{1/2}W +\frac{\alpha}{\beta}W'\right)$ is a conditional GMC over $\big(W,(\beta^2-\alpha^2)^{1/2}Y_{\mathbf{M}_{\alpha } }\big)$ with conditional expectation $\mathbf{M}_{\alpha }(W')$.

\end{itemize}

\end{example}

\subsection{Main theorem}

In the theorem statement below, $\Gamma$ denotes the space of directed paths on a diamond fractal defined in the next section, and $(\Gamma,\mathbf{M}_r)$ is a random measure  whose law arises as a continuum limit of disordered Gibbs measures on discrete models for random polymers~\cite[Section 2.6]{Clark4}.

\begin{theorem}\label{ThmMain} Let the one-parameter family of laws for random measures $(\mathbf{M}_r)_{r\in \R}$ on the  space $\Gamma$ be defined as in~\cite[Theorem 2.12]{Clark4} (restated below in  Theorem~\ref{ThmExist}). For a fixed $r\in \R$, let $W$ be a standard Gaussian random vector in $\mathcal{H}$ that is independent of $\mathbf{M}_r$.   There is a compact operator $Y_{  \mathbf{M}_r }:\mathcal{H} \rightarrow L^2(\Gamma,\mathbf{M}_r)  $ depending  measurably on  $ \mathbf{M}_r$ for which (i)-(iii) below  hold for any $a\in \R_+$.
\begin{enumerate}[(i)]

\item The operator  $\sqrt{a}Y_{\mathbf{M}_r }$ is a.s.\  a randomized shift.  The operator $\textbf{T}_{\mathbf{M}_r}:L^2(\Gamma,\mathbf{M}_r)\rightarrow     L^2(\Gamma,\mathbf{M}_r) $ defined by $\textbf{T}_{\mathbf{M}_r}:=Y_{\mathbf{M}_r}Y_{\mathbf{M}_r}^*$ is a.s.\ Hilbert-Schmidt but not trace class and has kernel $T(p,q)$ not depending on $\mathbf{M}_r$.

\item There is a unique conditional GMC $\mathds{M}_{r,a}$  over $(W,\sqrt{a} Y_{\mathbf{M}_r})$ with conditional expectation  $\mathbf{M}_r$.

\item The random measure $(\Gamma,\mathds{M}_{r,a})$ is equal in law to  $(\Gamma, \mathbf{M}_{r+a}) $.

\end{enumerate}

\end{theorem}

\section{A diamond fractal, its  path space, and a critical continuum model}\label{SecReview}

Sections~\ref{SecDHL}-\ref{SecHausdorff} review the construction from~\cite{Clark2} of a diamond fractal, which we refer to as the \textit{diamond hierarchical lattice} (DHL), along with its space of directed paths.  Section~\ref{SecCDRP} outlines the properties of the continuum random polymer measures $(\mathbf{M}_r)_{r\in \R}$ referred to in Theorem~\ref{ThmMain}. The presentation in points (A)-(V) below is intended to be scannable and readily referred back to, and a reader familiar with~\cite{Clark4} can skip to Section~\ref{SecThm}.\vspace{.15cm}

The DHL $D^{b,s}$, which depends on  a branching number $b\in \{2,3,\ldots\}$ and a segmenting number $s\in \{2, 3,\ldots\}$, defines a space $\Gamma^{b,s}$ of directed pathways between opposing nodes $A$ and $B$; see the figure below for a depiction of the  diamond fractal's self-similarity in the case of  $(b,s)=(2,3)$. A directed pathway is an isometrically embedded copy of the unit interval $[0,1]$ with $0\equiv A$ and $1\equiv B$.  
\begin{figure}[hbt!]
\centering
\includegraphics[scale=.6]{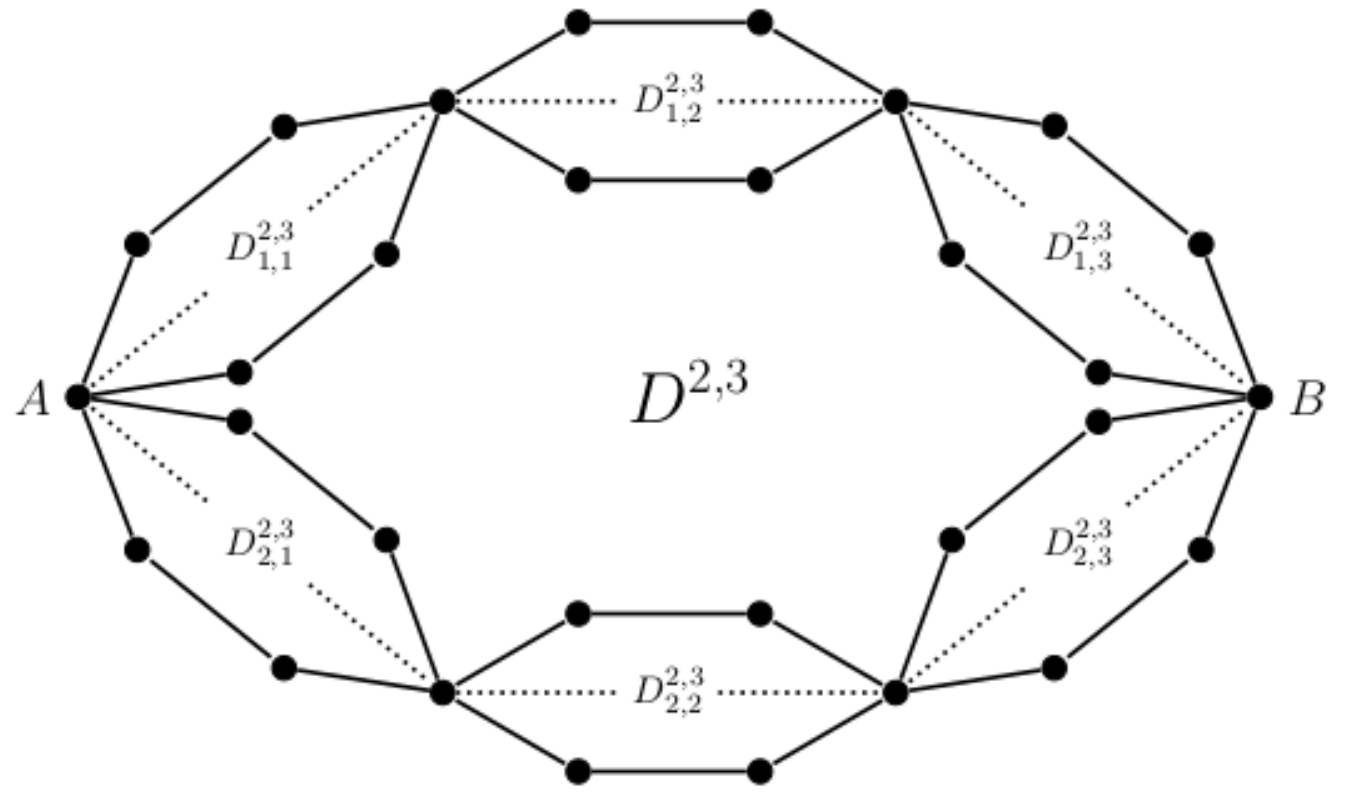}
\caption{The diamond fractal $D^{2,3}$  embeds shrunken copies $D^{2,3}_{i,j}$ of itself corresponding to each $(i,j)\in \{1,2\}\times \{1,2,3\}$.  The path space $\Gamma^{2,3}$ is canonically soluble as $\bigcup_{i=1}^2\bigtimes_{j=1}^3\Gamma^{2,3}$ through three-fold concatenation of paths crossing the subcopies of $D^{2,3}$.}
\end{figure}

\subsection{The DHL and its space of directed paths }\label{SecDHL}

\noindent \textbf{(A) Sequences:} Given  $b,s\in \{2,3,\ldots\}$, define  $$\mathcal{D}^{b,s}\,:=\,\big( \{1,\ldots, b\}\times\{1,\ldots, s\}\big)^{\infty}\,, $$ i.e.,  the set of sequences of ordered pairs $x=\{(b_k, s_k)  \}_{k\in \mathbb{N}}$, where $b_k\in \{1,\ldots, b\}$ and $s_{k}\in \{ 1,\ldots s\}$.  The DHL, $D^{b,s}$, is defined  as  an equivalence relation on $\mathcal{D}^{b,s}$,
\begin{align}\label{Equiv} D^{b,s}\,:=\, \mathcal{D}^{b,s}/\big(x,y\in \mathcal{D}^{b,s}\text{ with }d_D(x,y)=0\big) \,,
\end{align}
 for a semi-metric $d_D: \mathcal{D}^{b,s}\times \mathcal{D}^{b,s}\longrightarrow [0,1]$ to be defined below.\vspace{.25cm}

\noindent \textbf{(B) The semi-metric:} Define the map $\widetilde{\pi}: \mathcal{D}^{b,s}\rightarrow [0,1]$ such that a sequence $x= \{(b_k^x,s_k^x)\}_{k\in \mathbb{N}}$ is assigned the value 
$$\widetilde{\pi}(x)\,:=\, \sum_{k=1}^{\infty} \frac{s_{k}^{x}-1}{ s^{k}}  \,, $$
in other terms, the number with base-$s$  decimal expansion  having  $k^{th}$ digit $s_k^x-1\in\{0,\ldots, s-1\}$.   Define  the extremal sets $$A\,:=\,\big\{x\in \mathcal{D}^{b,s}\,\big|\, \widetilde{\pi}(x)=0 \big\}\hspace{1cm}\text{and}\hspace{1cm}B\,:=\,\big\{x\in \mathcal{D}^{b,s}\,\big|\, \widetilde{\pi}(x)=1 \big\}\,. $$
   For $x,y\in\mathcal{D}^{b,s}$ we write $x\updownarrow y$ if  $x$ or $y$ belongs to one of the sets $A$, $B$ or if the sequences of pairs $\{(b_k^x,s_k^x)\}_{k\in \mathbb{N}}$ and  $\{(b_k^y,s_k^y)\}_{k\in \mathbb{N}}$ defining $x$ and $y$, respectively, have their first disagreement at an $s$-component value, i.e., there exists an $n\in \mathbb{N}$ such that $  b_k^x=b_k^y$ for all $ 1\leq k \leq n $ and  $ s_n^x\neq s_n^y$.  Intuitively,  $x\updownarrow y$ means that there is directed path passing over both $x$ and $y$.
  We define the semi-metric $d_D$ in terms of $\widetilde{\pi}$    as
\begin{align}\label{DefSemiMetric}
d_D(x,y)\,:=\,\begin{cases} \quad  \quad  \big|\widetilde{\pi}(x)-\widetilde{\pi}(y)\big|   & \quad\text{if } x\updownarrow y,  \\ \,\, \displaystyle \inf_{z \in \mathcal{D}^{b,s},\, z\updownarrow x,\, z\updownarrow y  }\Big( d_D(x,z)+d_D(z,y) \Big)  & \quad  \text{otherwise.} \end{cases}    
\end{align}
 The semi-metric $d_D(x,y)$  takes values $\leq 1$ since, by definition, $z\updownarrow x $  and $z\updownarrow y $ for any $z\in A$ or $z\in B$, and thus $d_D(x,y)\leq \min\big(\widetilde{\pi}(x)+\widetilde{\pi}(y) ,2-\widetilde{\pi}(x)-\widetilde{\pi}(y)\big)$.\vspace{.2cm}

\noindent \textbf{(C) Vertex set:} Let $E^{b,s}$ denote the set of points $x\in D^{b,s}$ that correspond through~(\ref{Equiv}) to a single-element equivalence class of $\mathcal{D}^{b,s}$.  The complement $V^{b,s}=D^{b,s}\backslash E^{b,s}$ is a countable, dense set. 
\vspace{.2cm}

\noindent \textbf{(D) Directed paths:} A \textit{directed path}   on $D^{b,s}$ is a continuous function $p:[0,1]\rightarrow D^{b,s}$ such that $\widetilde{\pi}\big( p(r)\big)=r $ for all $r\in[0,1]$.  
Thus the path moves  at a constant speed from $A$ to $B$.  We can measure the distance between paths using the uniform metric:
$$    d_\Gamma\big(p_{1},p_{2}\big) \, = \, \max_{0\leq r\leq 1}d_D\big( p_{1}(r), p_{2}(r)    \big)  \,\,\,\, \text{for}\,\,\,\, p_{1},p_{2}\in \Gamma^{b,s}   \,.     $$
 Paths  cross over  $V^{b,s}$ at the countable set $\mathcal{V}\subset [0,1]$ of times $t$ of the form $t=\frac{k}{s^n}$ for $k,n\in \mathbb{N}_0$.

\subsection{Cylinder sets and uniform measures}\label{SecCylinder}

\noindent \textbf{(E) Shift maps:} Define the shift maps  $S_{i,j}:\mathcal{D}^{b,s}\rightarrow \mathcal{D}^{b,s}$ for $(i,j)\in \{1,\ldots, b\}\times  \{1,\ldots, s\}$ that send a sequence $x\in \mathcal{D}^{b,s}$ to a shifted sequence $y=S_{i,j}(x)$ having initial term $(i,j)$, i.e., $\{(b_k^x,s_k^x)\}_{k\in \mathbb{N}}$ is mapped to $\{(b_k^y,s_k^y)\}_{k\in \mathbb{N}}$ for $(b_1^y,s_1^y)=(i,j)$ and $(b_k^y,s_k^y)=(b_{k-1}^x,s_{k-1}^x)$ for $k\geq 2$.  The maps $S_{i,j}$ are well-defined on $D^{b,s}$  and have the contractive property
\begin{align}\label{Contract}
 d_D\big(S_{i,j}(x),S_{i,j}(y)\big)\,=\,\frac{1}{s} d_D(x,y) \,\,\,\,\text{for}\,\,\,\, x,y\in D^{b,s} \,. 
\end{align}

\noindent \textbf{(F) Cylinder subsets of the DHL:} For a length-$n$  sequence of pairs $( b_k ,s_k)\in \{1,\ldots b\}\times \{1,\ldots, s\}$, define the following subset of $E^{b,s}$:
\begin{align*}
C_{ (b_1,s_1)\times \cdots \times (b_n,s_n)}  \,:=\,   S_{ b_1,s_1 }\circ \cdots \circ S_{b_n ,s_n }\big(E^{b,s} \big)\,.
\end{align*}
The collection of cylinder sets $C_{ (b_1,s_1)\times \cdots \times (b_n,s_n)}$  corresponding to length-$n$ sequences of pairs forms a partition of $E^{b,s}$ that we denote by $E^{b,s}_n$.  Let  $\langle x\rangle_n \in E_n^{b,s}$ denote the equivalence class   of $x\in E^{b,s}$.  \vspace{.2cm}

\noindent \textbf{(G) Uniform measure on the DHL:} The Borel  $\sigma$-algebra, $\mathcal{B}_D$, of $(D^{b,s},d_D)$ is  generated by subsets of $V^{b,s}$ and  elements in  $\bigcup_{k=0}^{\infty}E_k^{b,s}$.  There is a unique normalized  measure $\nu$ on $(D^{b,s},\mathcal{B}_D )$ that assigns the countable set  $V^{b,s}$ measure zero and cylinder sets $e\in E^{b,s}_{n}$ measure $\nu(e)=|E_n^{b,s}  |^{-1}= (bs)^{-n}$.\vspace{.2cm}

\noindent \textbf{(H) Cylinder sets for directed paths:} For any $n\in \mathbb{N}$, a path $p\in \Gamma^{b,s}$ determines a function $[p]_n:\{1,\ldots, s^n\}\rightarrow E_n^{b,s}$, where $[p]_n(k):= \langle p(t) \rangle_n$ for $t\in (\frac{k-1}{s^n},\frac{k}{s^n})\backslash \mathcal{V}$.     The map $p \mapsto [p]_n$ determines a partition  $\Gamma^{b,s}_n$ of $\Gamma^{b,s}$, and we interpret our notation flexibly by identifying $[p]_n$  with the equivalence class of $p$ in $\Gamma^{b,s}_n$.\vspace{.2cm}

\noindent \textbf{(I) Uniform measure on directed paths:} The Borel $\sigma$-algebra, $\mathcal{B}_\Gamma $, on  $(\Gamma^{b,s}, d_\Gamma)$ is generated by the collection of cylinder sets $\cup_{n=1}^\infty \Gamma^{b,s}_n $.
 There is a unique measure $\mu$ on $(\Gamma^{b,s},\mathcal{B}_\Gamma)$ satisfying $\mu(\mathbf{p})=|  \Gamma^{b,s}_n |^{-1}  $ for all $n\in \mathbb{N}$ and  $\mathbf{p}\in \Gamma^{b,s}_n$.  The \textit{uniform measure} on  $\Gamma^{b,s}$  refers to the triple      $\big(\Gamma^{b,s},\mathcal{B}_\Gamma,  \mu\big)$.


\subsection{Hausdorff dimension considerations}\label{SecHausdorff}

\noindent \textbf{(J) Hausdorff dimension of the DHL:} The shift maps $S_{i,j}$ take $D^{b,s}$ to  shrunken, embedded copies  of $D^{b,s}$, denoted by $D^{b,s}_{i,j}$, whose pair-wise intersections are either empty or consist of a single point.   The maps $S_{i,j}$ are the similitudes of the fractal $D^{b,s}$, and since there  are $b s$ similitudes satisfying the contraction property~(\ref{Contract}), $D^{b,s}$ has Hausdorff dimension $(\log b+\log s)/\log s$; see~\cite[Section~11.3]{Folland}.  In particular when $b=s$ the DHL has Hausdorff dimension two.
 \vspace{.2cm}

 \noindent \textbf{(K) Hausdorff dimension of path intersection sets when $\mathbf{b<s}$:} For $p,q\in \Gamma^{b,s}$, define  the set of intersection times: $I_{p,q}:=\big\{t\in [0,1]\,\big|\,p(t)=q(t)\big\}$.  Assume  $b<s$, i.e., that the DHL has less branching than segmenting.  If the paths $p,q\in \Gamma^{b,s}$ are independently chosen uniformly at random, i.e., according to the product measure   $\mu\times \mu$, then one of the following events a.s.\ occurs:
 $$ \text{(i) $I_{p,q}$ is a finite set}\hspace{1cm}\text{or}\hspace{1cm} \text{(ii) $I_{p,q}$ has Hausdorff dimension $\frak{h}=\frac{\log s-\log b  }{\log s   }$}\,,     $$
and the probability $\frak{p}_{b,s}\in (0,1)$ of (ii) is a fixed point for the function $M(x)=\frac{1}{b}\big[1-(1-x)^s\big]$.  \vspace{.2cm}

 \noindent \textbf{(L) Trivial path intersections in the critical case $\mathbf{b=s}$:}  
In the Hausdorff dimension two case of the DHL ($b=s$), the product measure $\mu\times \mu$ is supported on pairs $(p,q)\in \Gamma^{b,b}\times \Gamma^{b,b}$ for which $I_{p,q}$ is finite. Thus (i) above occurs with probability one, and there is no chance that a nontrivial form of (ii) occurs in the sense of  $I_{p,q}$ being uncountably infinite despite having Hausdorff dimension zero. This contrasts with the family of random measures $(\Gamma^{b,b},\mathbf{M}_r)$ summarized in Section~\ref{SecCDRP} for which $(\Gamma^{b,b}\times \Gamma^{b,b}, \mathbf{M}_r\times\mathbf{M}_r)$ a.s.\ assigns  positive measure to the set of pairs $(p,q)$ with nontrivial intersection sets.

\subsection{Continuum random polymer measures  in the critical case of the DHL}\label{SecCDRP}
 
Next we will outline the defining properties for  the family of random measures $(\mathbf{M}_r)_{r\in \R}$ introduced in~\cite{Clark4}.  For the remainder of the article we will focus only on the critical case $b=s$ of the DHL and maintain $b\in \{2,3,\ldots\}$ as an underlying parameter that will be removed as a superscript from all DHL-related notations: $ D\equiv D^{b,b}$, $\Gamma\equiv \Gamma^{b,b}$, $ E\equiv E^{b,b}$, $ V\equiv V^{b,b}\,$.  The expectation symbol $\mathbb{E}$ will always refer to the underlying probability space $\big(\Omega, \mathscr{F} ,\mathbb{P})$ on which $\mathbf{M}_r$ is defined.  \vspace{.2cm}

\noindent  \noindent \textbf{(M) Disordered measure on the path space:} The following theorem is from~\cite{Clark4}. The uniquess of the family of laws with properties (I)-(IV) will be used to prove part (iii) of Theorem~\ref{ThmMain}.

\begin{center}
\begin{minipage}{.9\linewidth}
\begin{theorem}[Theorem 2.12 of~\cite{Clark4}]\label{ThmExist}
There is a unique   family of laws for random measures $(\mathbf{M}_{r})_{r\in \R}$  on the path space, $\Gamma$, of $D$  satisfying the properties (I)-(IV) below.

\begin{enumerate}[(I)]

\item The expectation of the measure $\mathbf{M}_{r}$ with respect to the underlying probability space is the uniform measure on paths, i.e., $\mathbb{E}[ \mathbf{M}_{r} ]=\mu$.

\item For a correlation measure $(\Gamma\times \Gamma, \upsilon_r) $ discussed below, we have that $\mathbb{E}[ \mathbf{M}_r \times  \mathbf{M}_r    ]= \upsilon_r $. 

\item For $m\in \{2,3,\ldots \}$, the $m^{th}$ centered moment of the total mass, $\mathbf{M}_r(\Gamma)$, is given by $R^{(m)}(r)$ for an increasing  function $R^{(m)}:\R\rightarrow \R_+$  that decays in proportion to $(-r)^{-\lceil m/2\rceil}$ as $r\rightarrow -\infty$ and grows without bound as $r\rightarrow \infty$.

\item Let $(\Gamma,\mathbf{M}_{r}^{(i,j)}) $ be independent copies of $(\Gamma,\mathbf{M}_{r}) $ corresponding to the first-generation embedded copies, $D_{i,j}$, of $D$.  There is equality in distribution of the random measures $
\mathbf{M}_{r+1 }\,\stackrel{d}{=}\, \frac{1}{b}\sum_{i=1}^b  \prod_{j=1}^b  \mathbf{M}_{r }^{(i,j)}$  under the identification $ \Gamma\,\equiv\, \bigcup_{i=1}^{b}\bigtimes_{j=1}^b  \Gamma 
$.

\end{enumerate}

\end{theorem}
\end{minipage}
\end{center}

\noindent  \noindent \textbf{(N) Basic properties of $\mathbf{M}_r$:} The random measures $\mathbf{M}_r$ are a.s.\  mutually singular with respect to $\mu$ although they a.s.\ assign positive measure  $\mathbf{M}_r(A)>0$ to every  open set  $A\subset\Gamma$.  As $r\rightarrow -\infty$, $\mathbf{M}_r$ converges to $\mu$  in the sense that for any $F\in L^2(\Gamma,\mu  )$ the random variable $\int_{\Gamma}F(p)\mathbf{M}_r(dp) $ converges in $L^2$ to the constant $\int_{\Gamma}F(p)\mu(dp) $.\vspace{.2cm}

\noindent  \noindent \textbf{(O) Total mass variance function:} The variance $R(r):=\textup{Var}\big(\mathbf{M}_{r}(\Gamma) \big)$ is a  continuous increasing function $R:\R\rightarrow \R_+$   satisfying the recursive relation:\, $ \frac{1}{b}\big[\big(1+R(r)\big)^b-1\big]= R(r+1) $ for all $r\in \R$ and having the following vanishing asymptotics  as $r\rightarrow -\infty$:
$$  R(r)\,=\,-\frac{ \kappa^2 }{ r }\,+\, \frac{ \kappa^2\eta\log(-r) }{ r^2 }\,+\,\mathit{O}\bigg( \frac{\log^2(-r)}{r^3} \bigg) \,\text{ for constants $\,\kappa:=\sqrt{\frac{2}{b-1}}\,$ and $\,\eta:=\frac{b+1}{3(b-1)}$}\,. $$

\noindent  \noindent \textbf{(P) Correlation measure:}  The correlation measure $(\Gamma\times \Gamma, \upsilon_r   )$ is determined by assigning products, $\mathbf{p}\times\mathbf{q}$, of cylinder sets $\mathbf{p},\mathbf{q}\in \Gamma_{n}$  weight $
 \upsilon_{r}(\mathbf{p}\times \mathbf{q}   )\,=\, \frac{1}{|\Gamma_n|^2} \big(1+R(r-n)\big)^{N_n(\mathbf{p},\mathbf{q}) }$, where $N_n(\mathbf{p},\mathbf{q})$ is the number of $k\in \{1,\ldots,b^n\}$ such that $\mathbf{p}(k)=\mathbf{q}(k)$.  The marginals of $\upsilon_r$ are both $(1+R(r))\mu$ although $ \upsilon_{r}$ is not absolutely continuous with respect to $\mu\times \mu$; see (S) below.  \vspace{.2cm}

\noindent  \noindent \textbf{(Q) Intersection-time kernel:}  The kernel $T(p,q)$ in the definition below is $\upsilon_r$-a.e.\ finite by~\cite[Theorem 2.28]{Clark4} and effectively measures the set of intersection times $I_{p,q}=\{ t\in [0,1]\,|\,  p(t)=q(t)\}$.
\begin{center}
\begin{minipage}{.9\linewidth}
\begin{definition}\label{DefKernel} For $p,q\in \Gamma$ define $N_n(p,q)$  as $N_n(\mathbf{p},\mathbf{q})$  for  $(\mathbf{p},\mathbf{q}) =\big([p]_n, [q]_n\big)$. We define $ T(p,q):= \lim_{n\rightarrow \infty}\frac{\kappa^2}{n^2}N_n(p,q)$ when the limit exists and  $ T(p,q):=\infty $  otherwise. 
\end{definition}
\end{minipage}
\end{center}

\noindent  \noindent \textbf{(R) Exponential moments of the intersection-time kernel:} For any $a,r\in \R$, the Radon-Nikodym derivative of $\upsilon_{r+a}$  with respect to $\upsilon_r$ is $\textup{exp}\{aT(p,q)\} $.  The exponential moments of $T(p,q)$ with respect to $\upsilon_r$ are thus finite with $
\int_{\Gamma\times \Gamma}\textup{exp}\{aT(p,q)\} \upsilon_r(dp,dq)=\upsilon_{r+a}(\Gamma\times\Gamma)=1+R(r+a)$. Moreover, since $\mathbb{E}[ \mathbf{M}_r\times  \mathbf{M}_r  ]=\upsilon_r$,   the exponential moments $\int_{\Gamma\times \Gamma}\textup{exp}\{aT(p,q)\} \mathbf{M}_r(dp)\mathbf{M}_r(dq)$ are also a.s. finite. \vspace{.2cm}

\noindent  \noindent \textbf{(S) Lebesgue decomposition of the correlation measure:} The Lebesgue decomposition of $\upsilon_r$ with respect to $\mu\times \mu$ has the form $
\upsilon_r=\mu\times \mu+R(r)\rho_r $,
where $\rho_r$ is a probability measure supported on the set of pairs $(p,q)\in \Gamma\times \Gamma$ with $ 0<T(p,q)<\infty$, and the product  $\mu\times \mu$  is supported on the set of pairs with $T(p,q)=0$ and, in fact, for which $N_n(p,q)$ in Definition~\ref{DefKernel} is zero for  large  $n\in \mathbb{N}$.\vspace{.2cm}

\noindent  \noindent \textbf{(T) The  product $\mathbf{M}_r\times \mathbf{M}_r  $:} For any $r\in \R$, the product  $\mathbf{M}_r\times \mathbf{M}_r  $ is a.s.\  supported on the set of pairs $(p,q)\in \Gamma\times \Gamma$ such that $T(p,q)<\infty$.  The following lemma implies that $\mathbf{M}_r$-a.e.\ path has nontrivial intersection set with a $\mathbf{M}_r$-nonnegligible portion of the path space $\Gamma$:
\begin{center}
\begin{minipage}{.9\linewidth}
\begin{lemma}[Theorem 2.28 of~\cite{Clark4}]\label{LemmaInt} Given $p\in \Gamma$ define $\mathbf{\widehat{s}}_p$ as the set of $q\in \Gamma$ such that $T(p,q)>0$.  The random measure $\mathbf{M}_r$ a.s.\ satisfies that $\mathbf{M}_r(\mathbf{\widehat{s}}_p)>0$ for  $\mathbf{M}_r$-a.e.\ $p\in \Gamma$. 
\end{lemma}
\end{minipage}
\end{center}

\vspace{.1cm}

\noindent  \noindent \textbf{(U) A Hilbert-Schmidt operator defined by the intersection-time kernel:} The following theorem characterizes the operator on $L^2(\Gamma,\mathbf{M}_r)$ defined by integrating against the kernel $T(p,q)$.
\begin{center}
\begin{minipage}{.9\linewidth}
\begin{theorem}[Theorem 2.42 in~\cite{Clark4}]\label{ThmKernel} For a.e.\ realization of $\mathbf{M}_r$, the linear map,  $\mathbf{T}_{ \mathbf{M}_r }$, on $L^2(\Gamma,\mathbf{M}_r)$ defined by $(\mathbf{T}_{ \mathbf{M}_r }\psi)(p)=\int_{\Gamma}T(p,q)\psi(q)\mathbf{M}_r(dq)   $ has the properties below.
\begin{enumerate}[(i)]
\item $\mathbf{T}_{ \mathbf{M}_r }$ is Hilbert-Schmidt but not traceclass.

\item $\mathbf{T}_{ \mathbf{M}_r }=\hat{Y}_{ \mathbf{M}_r }\hat{Y}_{ \mathbf{M}_r }^*    $ for a compact  operator $\hat{Y}_{ \mathbf{M}_r }:L^2(D,\vartheta_{\mathbf{M}_r })\rightarrow L^2(\Gamma,\mathbf{M}_r)$ depending measurably on $\mathbf{M}_r $, where $(D,\vartheta_{\mathbf{M}_r })$ is a  Borel measure having total mass, $\vartheta_{\mathbf{M}_r }(D)$, equal to $\int_{\Gamma\times \Gamma}T(p,q)\mathbf{M}_r(dp)\mathbf{M}_r(dq)$ and expectation $\mathbb{E}\big[ \vartheta_{\mathbf{M}_r }]=R'(r)\nu$.

\item The  operator $\hat{Y}_{ \mathbf{M}_r }$ sends the constant  function $1_{D}$ to $\mathbf{t}_{ \mathbf{M}_r }(p):= \int_{\Gamma}T(p,q)\mathbf{M}_r(dq)$, which is  $\mathbf{M}_r$-a.e.\ positive for a.e.\ realization of $\mathbf{M}_r$  by Lemma~\ref{LemmaInt}.

\end{enumerate}

\end{theorem}

\end{minipage}

\end{center}

\noindent  \noindent \textbf{(V) A remark on renormalization symmetry:} The following proposition describes the hierarchical  relationship of the family of random measures $(\vartheta_{\mathbf{M}_{r}} )_{r\in \R} $ and  compact operators  $(\hat{Y}_{\mathbf{M}_{r}})_{r\in \R} $.
\begin{center}
\begin{minipage}{.9\linewidth}
\begin{proposition}[Sections 2.8 \& 2.9 of~\cite{Clark4}]\label{PropHierSymm}  Let $\{\mathbf{M}_{r}^{(i,j)}\}_{i,j\in\{1,\ldots, b\}}$ be a family of independent copies   of $\big(\Gamma, \mathbf{M}_{r}\big)$ such that  $\mathbf{M}_{r+1}=\frac{1}{b}\sum_{i=1}^b\prod_{j=1}^b\mathbf{M}_{r}^{(i,j)} $ in the sense of property (IV) of Theorem~\ref{ThmExist}.  The measure  $(D,\vartheta_{\mathbf{M}_{r+1} })$ has decomposition
\begin{align*}
\vartheta_{\mathbf{M}_{r+1} }\,=\,\frac{1}{b^2}\bigoplus_{1\leq i,j\leq b  } \bigg(\prod_{\ell \neq j}\mathbf{M}_{r}^{(i,\ell)}(\Gamma)    \bigg)^2  \vartheta_{\mathbf{M}_{r}^{(i,j)} } \hspace{.2cm}\text{under the identification}\hspace{.2cm} D\,\equiv\, \bigcup_{1\leq i,j\leq b}D_{i,j}\,,
\end{align*}
where $D_{i,j}$ is a copy of $D$ corresponding to the first-generation sub-copy of the DHL situated at  the $j^{th}$ segment along the  $i^{th}$ branch.  Similarly,  $\hat{Y}_{\mathbf{M}_{r+1}}$ decomposes as $ \big(\hat{Y}_{\mathbf{M}_{r+1}}\hat{\phi}\big)(p)\,=\,\sum_{j=1}^b   \big(\hat{Y}_{\mathbf{M}_{r}^{(i,j)}}\hat{\phi}^{(i,j)}\big)(p_{j})$,   
where $\hat{\phi}^{(i,j)}\in L^2\big(D_{i,j}, \vartheta_{\mathbf{M}_{r}^{(i,j)} } \big)$ are the components of  $\hat{\phi}\in L^2\big(D, \vartheta_{\mathbf{M}_{r+1} } \big)$, and we identify $p\in \Gamma$ with the $(b+1)$-tuple  $(i; p_1,\ldots, p_b)\in \{1,\ldots,b\}\times \Gamma^b$.\footnotemark

\end{proposition}
\end{minipage}

\end{center}
\footnotetext{Note that the indentification  $\Gamma\equiv \{1,\ldots,b\}\times \Gamma^b$ is equivalent to   $\Gamma\equiv\bigcup_{i=1}^b\Gamma^b$  in (IV) of Theorem~\ref{ThmExist}.}

\section{Proof of Theorem~\ref{ThmMain}}\label{SecThm}

Once the relevant definitions are formed, parts (i) and (ii) of Theorem~\ref{ThmMain} (restated in Proposition~\ref{CorRandomShift} below)  follow easily from results in  Section~\ref{SecGMCDef} and Section~\ref{SecCDRP}.  The proof of part (iii) of Theorem~\ref{ThmMain}  will depend on the uniqueness of the family of laws $(\mathbf{M}_{r})_{r\in \R}$ satisfying properties (I)-(IV) of Theorem~\ref{ThmExist} and on the uniqueness of  subcritical GMC.

\subsection{Constructing the Gaussian field and the conditional GMC}

\begin{definition}\label{DefY} For $r\in \R$, let the measure  $(D,\vartheta_{\mathbf{M}_r})$ and the  map $\hat{Y}_{ \mathbf{M}_r }:L^2(D,\vartheta_{\mathbf{M}_r})\rightarrow L^2(\Gamma,\mathbf{M}_r)$ be defined as in Theorem~\ref{ThmKernel}. Moreover, choose some isometric linear map $\hat{U}_{\mathbf{M}_r}: L^2(D,\vartheta_{\mathbf{M}_r})\rightarrow \mathcal{H}$    depending measurably on $\mathbf{M}_r$.\footnote{$\hat{U}_{\mathbf{M}_r}$ is trivial to construct using Gram-Schmidt to find an orthonormal basis of $L^2(\Gamma,\mathbf{M}_r)$; see Appendix~\ref{SecA}.}   We define $Y_{\mathbf{M}_r}:\mathcal{H}\rightarrow L^2(\Gamma,\mathbf{M}_r)$  as $Y_{\mathbf{M}_r}:=  \hat{Y}_{\mathbf{M}_r} \hat{U}_{\mathbf{M}_r}^*$.

\end{definition}

\begin{definition}\label{DefFields}  Let $\hat{U}_{\mathbf{M}_r}$ and  $Y_{\mathbf{M}_r}$ be defined as in Definition~\ref{DefY} and   $W:\mathcal{H}\rightarrow L^2(\Omega, \mathscr{F},\mathbb{P})  $ be a standard Gaussian random vector   independent of $\mathbf{M}_r$. 
 For a  realization of $(\Gamma, \mathbf{M}_r)$, define the Gaussian fields

\begin{enumerate}[(i)]

\item  $\hat{ W }_{\mathbf{M}_r}:L^{2}(D, \vartheta_{\mathbf{M}_r})\rightarrow  L^2(\Omega, \mathscr{F},\mathbb{P})$ as  $ \hat{ W }_{\mathbf{M}_r}:= W\hat{U}_{\mathbf{M}_r}   $ and

\item $\mathbf{W}_{\mathbf{M}_r}:L^{2}(\Gamma, \mathbf{M}_{r})\rightarrow L^2(\Omega, \mathscr{F},\mathbb{P})$ as  $\mathbf{W}_{\mathbf{M}_r}= W Y_{\mathbf{M}_r}^*    $.

\end{enumerate}

\end{definition}

\begin{remark}\label{RemarkCorrKernel} For a.e.\ realization of $\mathbf{M}_{r}$, the operator $\hat{ W }_{\mathbf{M}_{r}}$ defines a white-noise field on $D$ with variance measure $\vartheta_{\mathbf{M}_r}$.  The operator $\mathbf{W}_{\mathbf{M}_r}\equiv\{\mathbf{W}_{\mathbf{M}_r}(p)\}_{p\in \Gamma}$ a.s.\ defines a  Gaussian field over $(\Gamma,\mathbf{M}_r)$ with covariance operator $\textbf{T}_{\mathbf{M}_r}=Y_{\mathbf{M}_r}Y_{\mathbf{M}_r}^*=\hat{Y}_{\mathbf{M}_r}\hat{Y}_{\mathbf{M}_r}^*$ having kernel $T(p,q)$. 

\end{remark}

\begin{proposition}\label{CorRandomShift}  Let $Y_{\mathbf{M}_r}:\mathcal{H}\rightarrow L^2(\Gamma,\mathbf{M}_r)$ be as in  Definition~\ref{DefY} and $W:\mathcal{H}\rightarrow L^2\big(\Omega, \mathscr{F},\mathbb{P}\big)$ be a standard Gaussian random vector independent of $\mathbf{M}_r$.   For any  $a\in \R_+$,
\begin{enumerate}[(i)]
\item 
 the operator $\sqrt{a}Y_{\mathbf{M}_r}$ is a randomized shift for a.e.\ realization of $\mathbf{M}_r$, and

\item  there is a unique conditional GMC $\mathds{M}_{r,a}$    over  $(W,\sqrt{a}Y_{\mathbf{M}_r})$ with conditional expectation $\mathbf{M}_r$.

\end{enumerate}

\end{proposition}
\begin{proof}Part (i): The covariance operator $Y_{\mathbf{M}_r}Y_{\mathbf{M}_r}^*=\hat{Y}_{\mathbf{M}_r}\hat{Y}_{\mathbf{M}_r}^*=\textbf{T}_{\mathbf{M}_r}  $ is a.s.\ Hilbert-Schmidt  with kernel $T(p,q)$ by Theorem~\ref{ThmKernel}.  For  any $a\in \R_+$,  the exponential moment $\int_{\Gamma\times \Gamma}\textup{exp}\{aT(p,q)\}\mathbf{M}_r(dp)\mathbf{M}_r(dq)$ is  finite for a.e.\ realization of $\mathbf{M}_r$  by the remarks in part (R) of Section~\ref{SecCDRP}. By Corollary~\ref{CorApprox} this implies that $\sqrt{a} Y_{\mathbf{M}_r}$ is a.s.\ a randomized shift. 
 \vspace{.3cm}

\noindent Part (ii):  The  compact operator  $\sqrt{a}Y_{\mathbf{M}_r}$ is a.s.\ a randomized shift by  part (i), and hence there is a unique conditional GMC $\mathds{M}_{r,a}\equiv \mathds{M}_{r,a}(  \mathbf{M}_r   , W )$ over $(W,\sqrt{a}Y_{\mathbf{M}_r})$ with conditional expectation $ \mathbf{M}_r$  by Corollary~\ref{CorShift}.

\end{proof}

The following proposition, which we do not use, is a remark about the random Gaussian fields $\mathbf{W}_{\mathbf{M}_r}$ and $\hat{ W }_{\mathbf{M}_r}$  defining bounded linear operators on   $ L^2(\Gamma,\mu) $ and $ L^2(D,\nu) $, respectively. 

\begin{proposition}\label{PropSub} If $\psi\in L^2(\Gamma,\mu) $ and $\hat{\phi}\in L^2(D,\nu) $, then $\psi \in L^2(\Gamma,\mathbf{M}_r)$ and  $\hat{\phi} \in L^2(D,\vartheta_{\mathbf{M}_r})$  for a.e.\ realization of $ \mathbf{M}_r$ since $\mathbb{E}[\mathbf{M}_r]=\mu$ and  $\mathbb{E}[\vartheta_{ \mathbf{M}_r}]=R'(r)\nu$. 
Moreover, the linear maps   $\hat{ W }:L^2(D,\nu)\rightarrow L^2\big(\Omega, \mathscr{F},\mathbb{P}\big)$ and $\mathbf{W}:L^{2}(\Gamma, \mu)\rightarrow  L^2\big(\Omega, \mathscr{F},\mathbb{P}\big)$ defined  through $\hat{ W }\equiv\hat{ W }_{\mathbf{M}_r}$ and $\mathbf{W}\equiv \mathbf{W}_{\mathbf{M}_r}$ have operator norms bounded by $ \sqrt{R'(r)}$.

\end{proposition}

\begin{proof} For $\psi \in L^2(\Gamma, \mu)$ notice that by integrating out the field we get the first equality below
\begin{align*}
\mathbb{E}\Big[\big(\mathbf{W}_{ \mathbf{M}_{r} }   (\psi) \big)^2\Big]\,=\,\mathbb{E}\bigg[\int_{ \Gamma\times \Gamma  }\psi(p) \psi(q) T(p,q) \mathbf{M}_{r}(dp) \mathbf{M}_{r}(dp)  \bigg]\,=\,\int_{ \Gamma\times \Gamma  }\psi(p) \psi(q) T(p,q) \upsilon_{r}(dp,dq)\,,
\end{align*}
where the second equality holds by (II) of Theorem~\ref{ThmExist}.
Since $ \psi(p) \psi(q)\leq  \frac{1}{2}\big(| \psi(p)|^2+ |\psi(q)|^2 \big) $ and  the marginals of $T(p,q)\upsilon_r(dp,dq)$ are equal to $R'(r)\mu$ as a consequence of the remarks in part (R) of Section~\ref{SecCDRP}, the above is bounded by $R'(r)\|\psi\|^2_{L^2(\Gamma,\mu)}  $.  Since $\mathbb{E}[\vartheta_{\mathbf{M}_r }]=R'(r)\nu$, the operator $\hat{W}_{ \mathbf{M}_{r} }  $ is a $\sqrt{R'(r)}$-multiple   of an isometry since
if $\hat{\phi}\in L^2(D,\nu)   $
$$\mathbb{E}\Big[\big(\hat{W}_{ \mathbf{M}_{r} }   (\hat{\phi}) \big)^2\Big]\,=\,\mathbb{E}\bigg[\int_{D}\big|\hat{\phi}(x)\big|^2\vartheta_{\mathbf{M}_r }(dx)   \bigg]\,=\, R'(r)\int_{D} \big|\hat{\phi}(x)\big|^2\nu(dx)\,=\,R'(r)\| \hat{\phi}\|^2_{L^2(D,\nu)} \,. $$

\end{proof}

\subsection{Renormalization and the conditional GMC}

The proof of part (iii) of Theorem~\ref{ThmMain} will involve verifying   properties~(I)-(IV) in Theorem~\ref{ThmExist} for the conditional GMC $\mathds{M}_{r-a,a}$.  In particular,  property (IV) requires us to develop some notation for working with the renormalization transforms.

\begin{definition}\label{DefMeasure}  Let $\overline{M}:=\big\{ M^{(i,j)} \big\}_{i,j\in \{1,\ldots, b\}}$ be a family of measures on $\Gamma$. 
 Define $\Upsilon\overline{M}$ as the measure on $\Gamma$ such that
 \begin{align*}
 \Upsilon\overline{M}\,:=\,\frac{1}{b}\sum_{i=1}^b\prod_{j=1}^b M^{(i,j)} \hspace{.5cm}  \text{through the identification} \hspace{.5cm} \Gamma\,\equiv\,\bigcup_{i=1}^b\bigtimes_{j=1}^b \Gamma \,.
 \end{align*}

\end{definition}

 \begin{remark}\label{RemarkHier}   If $\mathbf{\overline{M}}_r:=\big\{\mathbf{M}_{r}^{(i,j)} \big\}_{i,j\in \{1,\ldots,b\}} $ is a family of i.i.d.\ copies of $(\Gamma,\mathbf{M}_{r})$, then property (IV) of Theorem~\ref{ThmExist} implies that $\mathbf{M}_{r+1}$ is equal in law to $\Upsilon \mathbf{\overline{M}}_r  $.  Similarly $\vartheta_{\Upsilon \mathbf{\overline{M}}_r }$ can be decomposed in terms of the measures  $\big(D, \vartheta_{\mathbf{M}_{r}^{(i,j)}} \big)$  as in Proposition~\ref{PropHierSymm}.
 
\end{remark}

\begin{definition}\label{DefCompositeStuffI} Let $W_{i,j  }$ be an independent  family of standard Gaussian random vectors in $\mathcal{H}$ indexed by $\{1,\ldots,b\}^2  $.  Note that the independence is equivalent to $\textup{Span}_{1\leq i,j\leq b}\big\{ \textup{Range}( W_{i,j  } ) \big\}  $ being a Gaussian subspace of $ L^2(\Omega, \mathscr{F},\mathbb{P})$ for which the spaces $\textup{Range}( W_{i,j})$ and $\textup{Range}( W_{I,J})$ are orthogonal when  $(i,j)\neq (I,J)$.  For $\overline{\mathcal{H}}:= \bigoplus_{1\leq i,j\leq b}\mathcal{H}$, we define $\overline{W}:\overline{\mathcal{H}}\rightarrow L^2(\Omega, \mathscr{F},\mathbb{P})   $ as the standard Gaussian random vector 
$$ \overline{W}\,:=\,\bigoplus_{1\leq i,j \leq b}W_{i,j  } \,.  $$

\end{definition}

\begin{definition}\label{DefCompositeStuffII}  Fix $r\in \R$ and 
 let $\overline{\mathbf{M}}_{r}=\big\{\mathbf{M}_{r}^{(i,j)}\big\}_{i,j\in\{1,\ldots,b\}}$ be a family of independent copies of the random measure $(\Gamma,\mathbf{M}_{r})$ and  $Y_{\mathbf{M}_{r}^{(i,j)}}:\mathcal{H}\rightarrow L^2\big(\Gamma,\mathbf{M}_{r}^{(i,j)}\big)$ be the corresponding copies of the operator $Y_{\mathbf{M}_{r}}$ defined as in (ii) of Definition~\ref{DefY}. For $\overline{\mathcal{H}}:= \bigoplus_{1\leq i,j\leq b}\mathcal{H}$, we
 define the linear operator $\overline{Y}_{\mathbf{\overline{M}}_r}:\overline{\mathcal{H}}\rightarrow L^2\big(\Gamma,\Upsilon\overline{\mathbf{M}}_{r}\big)    $ to act on $ \phi=\oplus_{1\leq i,j\leq b} \phi^{(i,j)}$ for $\phi^{(i,j)}\in \mathcal{H}$ as
$$ \big(\overline{Y}_{\mathbf{\overline{M}}_r}\phi\big)(p)\,:=\,     \sum_{j=1}^{b}\big(Y_{\mathbf{M}_{r}^{(i,j)}}\phi^{(i,j)} \big)(p_j) \,, $$
where $p\in \Gamma$ is identified with the $(b+1)$-tuple  $p=(i; p_1,\ldots, p_b)\in \{1,\ldots, b\}\bigtimes_{j=1}^b \Gamma$.

\end{definition}

\begin{lemma}\label{CorHier} Fix $a\in \R_+$ and $r\in \R$. Let $\overline{Y}_{\mathbf{\overline{M}}_r}$ be defined through a family of random measures $\overline{\mathbf{M}}_{r}$  as in Definition~\ref{DefCompositeStuffII}.  Moreover, let $\overline{W}$ be a copy of the standard Gaussian random vector in $\overline{\mathcal{H}}$ defined as in Definition~\ref{DefCompositeStuffI} and independent of $\overline{\mathbf{M}}_{r}$.

\begin{enumerate}[(i)]
\item For a.e.\ realization of $\overline{\mathbf{M}}_{r}$, the operator $\sqrt{a}\overline{Y}_{\mathbf{\overline{M}}_r}:\overline{\mathcal{H}}\rightarrow L^2(\Gamma,\Upsilon\overline{\mathbf{M}}_{r})  $ is a randomized shift.  

\item The conditional GMC  $\mathds{M}_{r,a}^{\Upsilon}$  over $\big(\overline{W},\sqrt{a}\overline{Y}_{\mathbf{\overline{M}}_r}\big) $ with conditional expectation $\overline{\mathbf{M}}_r$ is equal in law to the conditional GMC on $\big(W, \sqrt{a}Y_{\mathbf{M}_{r+1}}\big) $ with conditional expectation $\mathbf{M}_{r+1}$.
\end{enumerate}

\end{lemma}

\begin{remark}\label{RemarkIsom}  Let $W$ be a standard Gaussian random vector in $\mathcal{H}$ and $Y$ be a generalized $\mathcal{H}$-valued function.  In the following proof, we will use that the Gaussian field defined by the pair $(W,Y)$ is the same as the Gaussian field defined by the  pair $(WU,YU)$ provided that $U:\hat{\mathcal{H}}\rightarrow  \mathcal{H}$ is a linear isometry with $\big( \textup{Null}(Y)\big)^{\perp}\subset \textup{Range}(U)$.  This can be understood  on a formal level through $\mathbf{W}(p):=\langle W,Y(p)\rangle_{\mathcal{H}}= \langle U^*W,U^* Y(p)\rangle_{\hat{\mathcal{H}}}  $.   
\end{remark}

\begin{proof}[Proof of Lemma~\ref{CorHier}] Part (i) follows as a corollary of (ii). To see  (ii), first recall that the random measure $(\Gamma,\Upsilon\overline{\mathbf{M}}_{r})$ is equal in law to $(\Gamma,\mathbf{M}_{r+1})$ by Remark~\ref{RemarkHier}.  Thus we need to argue that the fields defined by $\big(\overline{W},\sqrt{a}\overline{Y}_{\mathbf{\overline{M}}_r}\big) $ and $\big(W, \sqrt{a}Y_{\mathbf{M}_{r+1}}\big) $ are equal in law  for a.e.\  realization of $\Upsilon\overline{\mathbf{M}}_{r}\equiv \mathbf{M}_{r+1}$.   By applying Remark~\ref{RemarkIsom} with the isometry $\hat{U}_{\mathbf{M}_{r+1}}$ in Definition~\ref{DefY}, we get that $\big(W, \sqrt{a}Y_{\mathbf{M}_{r+1}}\big) $ defines the same Gaussian field as $\big(\hat{W}_{\mathbf{M}_{r+1}}, \sqrt{a}\hat{Y}_{\mathbf{M}_{r+1}}\big) $.  Similarly, $\big(\overline{W},\sqrt{a}\overline{Y}_{\mathbf{\overline{M}}_r}\big) $ defines the same field as  $\big( \hat{W}_{\overline{\mathbf{M}}_{r}  },\sqrt{a} \hat{Y}_{\overline{\mathbf{M}}_{r}  }\big) $, where $ \hat{W}_{\overline{\mathbf{M}}_{r}  }:=\overline{W}\hat{U}_{ \overline{\mathbf{M}}_{r} }$ and $\hat{Y}_{\overline{\mathbf{M}}_{r}  }:=\overline{Y}_{\mathbf{\overline{M}}_r}\hat{U}_{ \overline{\mathbf{M}}_{r} }$ for the isometry $\hat{U}_{ \overline{\mathbf{M}}_{r} }$ defined as the direct sum of isometric maps $\hat{U}_{ \mathbf{M}_{r}^{(i,j)} }:L^2(D_{i,j},\vartheta_{  \mathbf{M}_{r}^{(i,j)}  } )\rightarrow \mathcal{H}$ (interpreted as copies of $\hat{U}_{ \mathbf{M}_{r}^{(i,j)} }$):
\begin{align*}\hat{U}_{ \overline{\mathbf{M}}_{r} }\,:=\,\bigoplus_{1\leq i,j\leq b} \hat{U}_{ \mathbf{M}_{r}^{(i,j)} }\,.
 \end{align*}
  The operator $\hat{W}_{\overline{\mathbf{M}}_{r}  }$ defines a white noise field  on $\overline{D}:=\bigcup_{1\leq i,j\leq b} D_{i,j}$ with variance measure $\overline{\vartheta}_{\overline{\mathbf{M}}_{r}   }:=\bigoplus_{1\leq i,j\leq b}\vartheta_{\mathbf{M}_r^{(i,j)}   }$ and $\hat{Y}_{\overline{\mathbf{M}}_{r}  }$     acts on $\hat{\phi}\in L^2\big(\overline{D},  \overline{\vartheta}_{\overline{\mathbf{M}}_{r}   }  \big)$ as
$$ \big(\hat{Y}_{\overline{\mathbf{M}}_{r}  }\hat{\phi}\big)(p)\,=\,\sum_{j=1}^b   \left(\hat{Y}_{\mathbf{M}_{r}^{(i,j)}}\hat{\phi}^{(i,j)}\right)(p_{j})  $$
for $p\equiv (i; p_1,\ldots, p_b)$ and $\hat{\phi}=\oplus_{1\leq i,j \leq b}\hat{\phi}^{(i,j)}  $ with $\hat{\phi}^{(i,j)} \in L^2\big( D_{i,j}, \vartheta_{  \mathbf{M}_{r}^{(i,j)} }  \big)$.  Since $\hat{W}_{\mathbf{M}_{r+1}}$ is a white noise field on $D$ with variance measure $\vartheta_{ \mathbf{M}_{r+1} }$, the identities in  Proposition~\ref{PropHierSymm} imply that the  field  $\big( \hat{W}_{\overline{\mathbf{M}}_{r}  },\sqrt{a} \hat{Y}_{\overline{\mathbf{M}}_{r}  }\big)  $ is equal in law to $\big(\hat{W}_{\mathbf{M}_{r+1}}, \sqrt{a}\hat{Y}_{\mathbf{M}_{r+1}}\big) $, which completes the proof.

\end{proof}

\subsection{Proof of part (iii) of  Theorem~\ref{ThmMain}}

\begin{proof}[Proof (iii) of Theorem~\ref{ThmMain}] Let the family of laws for random measures  $( \mathbf{M}_{r})_{r\in \R}$ on $\Gamma$ be defined as in Theorem~\ref{ThmExist}.  For $r\in \R$ and $a\in \R_+$,  define   $\mathds{M}_{r-a,a}$ as the conditional GMC over $\big(W,\sqrt{a}Y_{\mathbf{M}_{r-a}}\big)$ with conditional expectation $ \mathbf{M}_{r-a}$, which exists uniquely by part (ii) of Proposition~\ref{CorRandomShift}.  By the uniqueness of the family of laws $(\mathbf{M}_{r})_{r\in \R}$ satisfying properties (I)-(IV) in Theorem~\ref{ThmExist}, it suffices for us to verify that (I)-(IV) hold for  the family of laws $(\mathds{M}_{r-a,a})_{r\in \R}$.

\vspace{.3cm}

\noindent \textbf{Property (I):} Since $ \mathds{M}_{r-a,a}$ is a conditional GMC with conditional expectation $\mathbf{M}_{r-a}$,
$$\mathbb{E}\big[ \mathds{M}_{r-a,a}\big] \,=\, \mathbb{E}\Big[\mathbb{E}\big[ \mathds{M}_{r-a,a} \,\big|\,   \mathbf{M}_{r-a} \big]\Big] \,=\, \mathbb{E}\big[   \mathbf{M}_{r-a} \big]\,=\,\mu\,,  $$
where the third equality is by  property (I) of Theorem~\ref{ThmExist} for $ \mathbf{M}_{r-a} $.\vspace{.3cm}

\noindent \textbf{Property (II):} For measurable $g:\Gamma\times \Gamma\rightarrow [0,\infty)$, we can insert a conditional expectation with respect to $ \mathbf{M}_{r-a} $ 
\begin{align*}
\mathbb{E}\bigg[\int_{\Gamma\times \Gamma}g(p,q)  \mathds{M}_{r-a,a}(dp)\mathds{M}_{r-a,a}(dq) \bigg]\,=\,&\mathbb{E}\Bigg[\mathbb{E}\bigg[\int_{\Gamma\times \Gamma}g(p,q)  \mathds{M}_{r-a,a}(dp)\mathds{M}_{r-a,a}(dq) \,\bigg|\, \mathbf{M}_{r-a}  \bigg]\Bigg]\,. 
\intertext{Since the Gaussian field $\mathbf{W}_{\mathbf{M}_{r-a}}\equiv (W,Y_{ \mathbf{M}_{r-a} }) $ generating the conditional GMC $\mathds{M}_{r-a,a}$ has kernel $T(p,q)$ when conditioned on  $ \mathbf{M}_{r-a} $, the above is equal to  } 
\,=\,&\mathbb{E}\bigg[\int_{\Gamma\times \Gamma}g(p,q)  e^{aT(p,q)}    \mathbf{M}_{r-a}(dp)\mathbf{M}_{r-a}(dq) \bigg]\,.
\intertext{By property (II) of Theorem~\ref{ThmMain} for $\mathbf{M}_{r-a}$,  }
\,=\,&\int_{\Gamma\times \Gamma}g(p,q)  e^{aT(p,q)}    \upsilon_{r-a}(dp,dq)\,.
\end{align*}
By the remark in (R) of Section~\ref{SecCDRP},  $\exp\{aT(p,q)\}$ is the Radon-Nikodym derivative of    $\upsilon_{r}$ with respect to $  \upsilon_{r-a}$, and thus the above is equal to  $\int_{\Gamma\times \Gamma}g(p,q)      \upsilon_{r}(dp,dq)$.  Since $g$ is an arbitrary nonnegative measurable function, $\mathbb{E}\big[  \mathds{M}_{r-a,a}\times  \mathds{M}_{r-a,a}\big]=\upsilon_{r}  $.

\vspace{.3cm}

\noindent \textbf{Property (III):}   As before, we begin with a conditional expectation with respect to $\mathbf{M}_{r-a}$ to write
\begin{align*}
\mathbb{E}\Big[\big(\mathds{M}_{r-a,a}(\Gamma)\big)^m \Big]\,=\,&\mathbb{E}\bigg[\mathbb{E}\Big[\big(\mathds{M}_{r-a,a}(\Gamma)\big)^m  \,\big|\, \mathbf{M}_{r-a}  \Big]\bigg]\\ 
\,=\,&\mathbb{E}\Bigg[ \int_{\Gamma^m} \textup{exp}\Bigg\{ a\sum_{1\leq i < j\leq m} T(p_i,p_j)     \Bigg\}\mathbf{M}_{r-a}(dp_1)\cdots \mathbf{M}_{r-a}(dp_m) \Bigg]\,,
\intertext{where we have applied Proposition~\ref{PropGMCmoments} to get an integral expression for the moments of $\mathds{M}_{r-a,a}(\Gamma)$ conditioned on $\mathbf{M}_{r-a} $. With the inequality $x_1\cdots x_n\leq \frac{1}{n}(x_1^n+\cdots +x_n^n)$  for nonnegative $x_j$, the above can be bounded by  }
\,\leq \,&\frac{2}{m(m-1)}\mathbb{E}\Bigg[ \int_{\Gamma^m} \sum_{1\leq i < j\leq m}\textup{exp}\bigg\{ a\frac{m(m-1)}{2}  T(p_i,p_j)     \bigg\}\mathbf{M}_{r-a}(dp_1)\cdots \mathbf{M}_{r-a}(dp_m) \Bigg]\,. \intertext{Expanding the sum and integrating out the $m-2$ variables absent from the integrand we get}
\,=\,&\mathbb{E}\bigg[ \int_{\Gamma\times \Gamma}\big(\mathbf{M}_{r-a}(\Gamma)  \big)^{m-2} \textup{exp}\bigg\{ a\frac{m(m-1)}{2}  T(p,q)     \bigg\}\mathbf{M}_{r-a}(dp)\mathbf{M}_{r-a}(dq) \bigg].
\intertext{Next by applying the inequality $xy\leq \frac{1}{2}x^{2}+\frac{1}{2}y^2$ for $x,y\geq 0$,  }
\,\leq \,&\mathbb{E}\bigg[\int_{\Gamma\times \Gamma} \left(\frac{1}{2} \big( \mathbf{M}_{r-a}(\Gamma)  \big)^{2m-4}\,+\,\frac{1}{2} e^{ a m(m-1)  T(p,q) }\right)\mathbf{M}_{r-a}(dp)\mathbf{M}_{r-a}(dq) \bigg]
 \\
\,= \,&\frac{1}{2}\mathbb{E}\Big[ \big(\mathbf{M}_{r-a}(\Gamma)  \big)^{2m-2}\Big]\,+\,\frac{1}{2}\mathbb{E}\bigg[\int_{\Gamma\times\Gamma} e^{ a m(m-1)  T(p,q) }\mathbf{M}_{r-a}(dp)\mathbf{M}_{r-a}(dq) \bigg]\,.
\end{align*}
The left term above is finite since the moments of $\mathbf{M}_{r-a}(\Gamma)$ are finite  by property (III) of Theorem~\ref{ThmExist}. By property (II) of Theorem~\ref{ThmExist} for $\mathbf{M}_{r-a}$, the right term above is equal to  
$$\int_{\Gamma\times\Gamma} e^{ a m(m-1)  T(p,q) }\upsilon_{r-a}(dp, dq)\,=\,\int_{\Gamma\times\Gamma} \upsilon_{r-a+a m(m-1)}(dp, dq) =\,1+R\big(r-a+am(m-1)\big)\,, $$   
where the first equality uses that $\upsilon_{t+u}$ has Radon-Nikodym derivative $\textup{exp}\{uT(p,q) \}$  with respect to $\upsilon_{t}$, and the second equality holds because $\upsilon_{t}$ has total mass $1+R(t)$.  Therefore the $m^{th}$ moment of $\mathds{M}_{r-a,a}(\Gamma)$ is finite.

\vspace{.3cm}

\noindent \textbf{Property (IV):} Let $\overline{\mathds{M}}_{r-a,a}:=\big\{\mathds{M}_{r-a,a}^{(i,j)}\big\}_{i,j\in\{1,\ldots, b\}}$
 be a  family of independent copies of the random measure $ \mathds{M}_{r-a,a}$.  We must show the equality in law  \begin{align}\label{Show}
\mathds{M}_{r+1-a,a}\,\stackrel{\mathsmaller{\mathcal{L}}}{=}\,\Upsilon\overline{\mathds{M}}_{r-a,a}\,, 
\end{align}  where $\Upsilon$ is defined as in Definition~\ref{DefMeasure}.  To construct the family $\overline{\mathds{M}}_{r-a,a}$,
\begin{itemize}
\item  let
 $\overline{\mathbf{M}}_{r-a}:= \big\{\mathbf{M}_{r-a}^{(i,j)}\big\}_{i,j\in\{1,\ldots, b\}}$ be a family of i.i.d.\ copies of  $(\Gamma,\mathbf{M}_{r-a})$,

 \item define $ Y_{\mathbf{M}_{r-a}^{(i,j)}} $ for each $(i,j)\in \{1,\ldots ,b\}^2$  as in Definition~\ref{DefY}, and 

\item let $W_{i,j}$ be independent standard Gaussian random vectors that are jointly independent of  $\overline{\mathbf{M}}_{r-a}$.

  \end{itemize}
  We define $ \mathds{M}_{r-a,a}^{(i,j)}\equiv \mathds{M}_{r-a,a}^{(i,j)}\big(\mathbf{M}_{r-a}^{(i,j)}, W_{i,j}  \big)$ as the conditional GMC over $\big(W_{i,j}, \sqrt{a} Y_{\mathbf{M}_{r-a}^{(i,j)}} \big)$ with conditional expectation $\mathbf{M}_{r-a}^{(i,j)}$.

Let $\overline{W} $ be defined in terms of the family $\{W_{i,j}\}_{i,j\in \{1,\ldots,b\}}$  as in Definition~\ref{DefCompositeStuffI} and  $ \overline{Y}_{\overline{\mathbf{M}}_{r-a}}$   be defined as in Definitions~\ref{DefCompositeStuffII}. Recall from Lemma~\ref{CorHier} that the conditional GMC $\mathds{M}_{r-a,a}^{\Upsilon}$ over  $\big(\overline{W},\sqrt{a} \overline{Y}_{\overline{\mathbf{M}}_{r-a}}\big) $ with conditional expectation $\overline{\mathbf{M}}_{r-a}$  is equal in law to the conditional GMC $\mathds{M}_{r+1-a,a}$ over  $\big(W,\sqrt{a} Y_{ \mathbf{M}_{r+1-a}}\big) $ with conditional expectation $\mathbf{M}_{r+1-a}$. Therefore, to prove~(\ref{Show}), it suffices to show that
\begin{align}\label{Deduce}
\mathds{M}_{r-a,a}^{\Upsilon}\,=\,\Upsilon\overline{\mathds{M}}_{r-a,a}
\end{align}
holds a.s.  By uniquess of the conditional GMC $\mathds{M}_{r-a,a}^{\Upsilon}$, we can deduce~(\ref{Deduce})  by verifying that   $\Upsilon\overline{\mathds{M}}_{r-a,a}$ fulfills conditions (I)-(III) in Definition~\ref{DefCGMC} for being a conditional GMC  over  $\big(\overline{W},\sqrt{a} \overline{Y}_{\overline{\mathbf{M}}_{r-a}}\big) $ with conditional expectation $ \Upsilon\overline{\mathbf{M}}_{r-a}$. We check these conditions below.\vspace{.25cm}

\noindent Condition (I): Notice that taking the conditional expectation of $\Upsilon\overline{\mathds{M}}_{r-a,a}$ with respect to $\overline{\mathbf{M}}_{r-a}$ yields
\begin{align*}
\mathbb{E}\big[ \Upsilon\overline{\mathds{M}}_{r-a,a}\,|\, \overline{\mathbf{M}}_{r-a} \big]\,=\,\mathbb{E}\Bigg[ \frac{1}{b}\sum_{i=1}^b \prod_{j=1}^b  \mathds{M}_{r-a,a}^{(i,j)} \,\bigg|\,  \overline{\mathbf{M}}_{r-a} \Bigg]\,=\,&\frac{1}{b}\sum_{i=1}^b \prod_{j=1}^b  \mathbb{E}\Big[   \mathds{M}_{r-a,a}^{(i,j)}\,\big|\, \mathbf{M}_{r-a}^{(i,j)} \Big]\\ \,=\,&\frac{1}{b} \sum_{i=1}^b \prod_{j=1}^b   \mathbf{M}_{r-a}^{(i,j)} \,=\,\Upsilon\overline{\mathbf{M}}_{r-a}\,,
\end{align*}
where the first and last equalities are understood through the  canonical identification $\Gamma\equiv \bigcup_{i=1}^b\bigtimes_{j=1}^b \Gamma$.  The second equality above uses that the pairs $\big(\mathbf{M}_{r-a}^{(i,j)}  ,  W_{i,j}\big)$ are independent. Since $\Upsilon\overline{\mathbf{M}}_{r-a}$ is a function of $\overline{\mathbf{M}}_{r-a} $, taking the conditional expectation of $\Upsilon\overline{\mathds{M}}_{r-a,a}$ with respect to $\Upsilon\overline{\mathbf{M}}_{r-a}$ would also yield $\Upsilon\overline{\mathbf{M}}_{r-a}$.  \vspace{.25cm}

\noindent  Condition (II):  By definition, the random measure  $\Upsilon\overline{\mathds{M}}_{r-a,a}$ is a function, $\Upsilon$, of the family $\overline{\mathds{M}}_{r-a,a}\equiv \big\{\mathds{M}_{r-a,a}^{(i,j)}\big\}_{i,j\in \{1,\ldots, b\}}$, whose components $\mathds{M}_{r-a,a}^{(i,j)} $ are functions of $\big(\mathbf{M}_{r-a}^{(i,j)}  , W_{i,j}    \big) $ by (II) of Definition~\ref{DefCGMC}.
Thus $\Upsilon\overline{\mathds{M}}_{r-a,a}$ is a function of $\big(\overline{\mathbf{M}}_{r-a}  , \overline{W}    \big) $.  To see that $\Upsilon\overline{\mathds{M}}_{r-a,a}$ is also a function of $\big(\Upsilon\overline{\mathbf{M}}_{r-a}  , \overline{W}    \big) $, notice that the information lost about a family of measures $\{M_{i,j}\}_{i,j\in \{1,\ldots,b\}   } $ through the operation $\Upsilon$ can be characterized as follows: for any  positive scalars $\{\lambda_{i,j}\}_{i,j\in \{1,\ldots,b\}   }$ with   $\prod_{j=1}^b \lambda_{i,j}=1$ for each $i$, we get
\begin{align}\label{Rel}
\Upsilon\{\lambda_{i,j}M_{i,j}\}_{i,j\in \{1,\ldots,b\}   }\,=\,\Upsilon\{M_{i,j}\}_{i,j\in \{1,\ldots,b\}   }\,.
\end{align}
In other terms, the family of measures $\{M_{i,j}\}_{i,j\in \{1,\ldots,b\}   }$ can be recovered from $\Upsilon\{M_{i,j}\}_{i,j\in \{1,\ldots,b\}   }$ up to such a family scalar multiples.  However, since  $\mathds{M}_{r-a,a}^{(i,j)}\big(\lambda_{i,j}\mathbf{M}_{r-a}^{(i,j)}  , W_{i,j}    \big)=\lambda_{i,j}\mathds{M}_{r-a,a}^{(i,j)}\big(\mathbf{M}_{r-a}^{(i,j)}  , W_{i,j}    \big) $ and $\Upsilon\overline{\mathds{M}}_{r-a,a}$ is also defined through $\Upsilon$,  the random measure $\Upsilon\overline{\mathds{M}}_{r-a,a}$ is a function of $\big(\Upsilon\overline{\mathbf{M}}_{r-a}  , \overline{W}    \big) $. \vspace{.25cm}

\noindent Condition (III): Notice that   a shift of the field $\overline{W}$ by $\phi \in \overline{\mathcal{H}}$ yields
\begin{align*}
\Upsilon\overline{\mathds{M}}_{r-a,a}\big(&\Upsilon\overline{\mathbf{M}}_{r-a}, \overline{W}+\phi, dp\big)\\  \,=&\,\frac{1}{b} \prod_{j=1}^b \mathds{M}_{r-a,a}^{(i,j)}\Big(\mathbf{M}_{r-a}^{(i,j)}, W_{i,j}+\phi^{(i,j)},dp_j\Big)\,,
\intertext{where  $\displaystyle \phi=\oplus_{1\leq i,j\leq b} \phi^{(i,j)}$ is the decomposition of $\phi$ in terms of $\phi^{(i,j)}\in  \mathcal{H}$, and we identify $p\equiv (i; p_1,\ldots,p_b)$ through the canonical correspondence $\Gamma\equiv \{1,\ldots, b\}\times \bigtimes_{j=1}^b \Gamma$.  Since the random measures $\mathds{M}_{r-a,a}^{(i,j)}$ are conditional GMCs over  $\big(W_{i,j}, \sqrt{a}Y_{\mathbf{M}_{r-a}^{(i,j)}}  \big)$, we can apply property (III) of Definition~\ref{DefCGMC} to write the above as    }
\,=&\,\frac{1}{b}\textup{exp}\Bigg\{\sqrt{a}\sum_{j=1}^b\Big(Y_{\mathbf{M}_{r-a}^{(i,j)}}\phi^{(i,j)}\Big) (p_j)   \Bigg\} \prod_{j=1}^b \mathds{M}_{r-a,a}^{(i,j)}\Big(\mathbf{M}_{r-a}^{(i,j)},W_{i,j},dp_j\Big)\,. \\
\intertext{Since $\big(\overline{Y}_{\overline{\mathbf{M}}_{r-a}}\phi \big)(p):=\sum_{j=1}^b \big(Y_{\mathbf{M}_{r-a}^{(i,j)}}\phi^{(i,j)}\big) (p_j)  $, the definition of $\Upsilon\overline{\mathds{M}}_{r-a,a}$ implies that the above can be written as  }
 \,=&\,\textup{exp}\left\{\sqrt{a}\big(\overline{Y}_{\Upsilon\overline{\mathbf{M}}_{r-a} }\phi)(p)  \right\}\Upsilon\overline{\mathds{M}}_{r-a,a}\big(\Upsilon\overline{\mathbf{M}}_{r-a},\overline{W},dp\big)\,,
\end{align*}
which verifies condition (III). \vspace{.2cm} 

We have  established that $\Upsilon\overline{\mathds{M}}_{r-a,a}$ is the conditional GMC over   $(\overline{W}, \sqrt{a}  \overline{Y}_{\overline{\mathbf{M}}_{r-a}  }   )$ with conditional expectation $\mathbf{M}_{r-a}^{\Upsilon}$.  By the observations  below~(\ref{Deduce}), the proof is complete.

\end{proof}


\section{An application of the GMC structure to strong disorder analysis}

The proof of the following proposition is a modification of the proof of~\cite[Theorem 1.15]{Clark2}, which  adapted an argument in~\cite{Lacoin} for discrete polymers.

\begin{proposition}\label{PropDisorder}
Let the family of  random measures $(\mathbf{M}_{r})_{r\in \R}$ be defined as in Theorem~\ref{ThmExist}. As $r\rightarrow \infty$ the total mass  $\mathbf{M}_{r}(\Gamma   )$ converges in probability to $0$.


\end{proposition}

\begin{proof} Since we are characterizing the limiting behavior of the random measures $\mathbf{M}_r$ as $r\rightarrow \infty$, it suffices to assume that $r>0$.  By part (iii) of Theorem~\ref{ThmMain}, we can construct $\mathbf{M}_r$ as the conditional GMC over $(W, \sqrt{r}Y_{\mathbf{M}_0})$ with conditional expectation $\mathbf{M}_0$, where $Y_{\mathbf{M}_0}$ is defined as in Definition~\ref{DefY}, and  $W$ is a standard Gaussian random vector in $\mathcal{H}$ that is independent of $\mathbf{M}_0$. The random measure $\mathbf{M}_r\equiv \mathbf{M}_r(\mathbf{M}_0, W)$ is a function of $(\mathbf{M}_0, W) $ satisfying
$$ \mathbb{E}\big[ \mathbf{M}_r\,| \,\mathbf{M}_0  \big]\,=\,\mathbf{M}_0 \hspace{1cm}                                                \text{and} \hspace{1cm} \mathbf{M}_r\big(\mathbf{M}_0, W+\phi, dp\big)\,=\,e^{\sqrt{r}(Y_{\mathbf{M}_0}\phi)(p)   }\mathbf{M}_r\big(\mathbf{M}_0, W,dp   \big) $$
for any $\phi\in \mathcal{H}$. To prove that $\mathbf{M}_{r}(\Gamma) $ converges in probability to zero, it suffices to show that the fractional moment  $F_{ \mathbf{M}_0}(r):= \mathbb{E}\big[ \big(\mathbf{M}_{r}(\mathbf{M}_0, W,\Gamma)\big)^{1/2}\,\big|\, \mathbf{M}_0 \big]$ converges to zero as $r\rightarrow \infty$ for a.e.\ realization of $\mathbf{M}_0$

Let $\phi\in \mathcal{H}$ be defined as $\phi:=- \hat{U}_{\mathbf{M}_0} 1_D$ where $\hat{U}_{\mathbf{M}_0}:L^2(D,\vartheta_{\mathbf{M}_0}) \rightarrow \mathcal{H}$ is the linear isometry in  Definition~\ref{DefY}.  Notice that 
\begin{align*}
\big(Y_{\mathbf{M}_0}  \phi\big)(p)\,=\,- \big(\widehat{Y}_{\mathbf{M}_0}  1_D\big)(p)\,=\,-\int_{\Gamma}T(p,q)\mathbf{M}_0(dq)\,=:\,-\mathbf{t}_{\mathbf{M}_0}(p)\,,  
\end{align*}
where the first equality is from the definition of $Y_{\mathbf{M}_0}$ in Definition~\ref{DefY}, and the second equality is from (iii) of Theorem~\ref{ThmKernel}.  For a.e.\ realization of $\mathbf{M}_0$, the function $\mathbf{t}_{\mathbf{M}_0}:\Gamma\rightarrow [0,\infty)$ is positive for $\mathbf{M}_0$-a.e.\ $p\in \Gamma$  as a consequence of Lemma~\ref{LemmaInt}.  In particular, this implies that we  have the a.s.\ convergence
\begin{align*}
\int_{\Gamma}e^{-\sqrt{r} \mathbf{t}_{\mathbf{M}_0}(p)  } \mathbf{M}_{0}(dp)\hspace{.5cm}\stackrel{ r\rightarrow \infty }{\longrightarrow    }  \hspace{.5cm}0\,.
\end{align*}

Define $\widehat{\mathbb{P}}$ as the  measure having derivative $
\frac{d\widehat{\mathbb{P}}}{d\mathbb{P}}\,=\,e^{\langle W, \phi\rangle  -\frac{1}{2}\|\phi\|^2     }$ with respect to $\mathbb{P}$, and
let $\widehat{\mathbb{E}}$ denote the expectation corresponding to $\widehat{\mathbb{P}}$. The Cauchy-Schwarz inequality yields that 
\begin{align}
 \mathbb{E}\bigg[\Big(\mathbf{M}_{r}\big(\mathbf{M}_0, W, \Gamma\big) \Big)^{\frac{1}{2}} \,\bigg|\, \mathbf{M}_0  \bigg]\,=\,&\widehat{\mathbb{E}}\bigg[\Big(\mathbf{M}_{r}\big(\mathbf{M}_0, W, \Gamma\big)\Big)^{\frac{1}{2}}  e^{-\langle W, \phi\rangle  +\frac{1}{2} \|\phi\|^2     }  \,\bigg|\, \mathbf{M}_0 \bigg]\nonumber \\  \,\leq \,&\widehat{\mathbb{E}}\big[\mathbf{M}_{r}\big(\mathbf{M}_0, W,\Gamma\big)    \,\big|\, \mathbf{M}_0\big]^{\frac{1}{2}}\widehat{\mathbb{E}}\bigg[\Big( e^{-\langle W, \phi\rangle +\frac{1}{2}   \|\phi\|^2 } \Big)^2  \bigg]^{\frac{1}{2}}\,.
\nonumber  
\intertext{Since $\widehat{\mathbb{E}}\big[ G(\mathbf{M}_0,W)   \,|\, \mathbf{M}_0\big]\,=\,\mathbb{E}\big[ G(\mathbf{M}_0, W+\phi)   \,|\, \mathbf{M}_0\big]$ for any  nonnegative measurable function $G$ of $(\mathbf{M}_0,W)$ and $\mathbf{M}_{r}\big(\mathbf{M}_0, W+\phi,dp\big)= e^{( Y_{\mathbf{M}_0} \phi)(p)}  \mathbf{M}_{r}\big(\mathbf{M}_0, W,dp\big)$ for $( Y_{\mathbf{M}_0} \phi)(p)=- \mathbf{t}_{\mathbf{M}_0}(p)  $, the above is equal to    }
  \, = \,&\mathbb{E}\bigg[ \int_{\Gamma}e^{-\sqrt{r} \mathbf{t}_{\mathbf{M}_0}(p)  } \mathbf{M}_{r}(dp)\,\bigg| \,   \mathbf{M}_0\bigg]^{\frac{1}{2}} \mathbb{E}\Big[ e^{-\langle W, \phi\rangle +\frac{1}{2}  \|\phi\|^2    }  \Big]^{\frac{1}{2}}\nonumber   \\  
    \, = \,&\bigg(\int_{\Gamma}e^{-\sqrt{r} \mathbf{t}_{\mathbf{M}_0}(p)  } \mathbf{M}_{0}(dp)   \bigg)^{\frac{1}{2}}  e^{\frac{1}{2} \|\phi\|^2    }  \,.\nonumber
\end{align}
This expression a.s.\ converges  to zero as $r\rightarrow \infty$ by the remark above.

\end{proof}

\begin{appendix}

\section{Proofs of Corollary~\ref{CorApprox} and Proposition~\ref{PropGMCmoments}}\label{SecZ}

Recall that $(\Gamma, \mathcal{B}_{\Gamma})$ is a  standard Borel measurable space and $\mu$ is a finite measure on $\Gamma$.  
 No generality is sacrificed by  taking $\mu$ to be a probability measure and  $\Gamma$ to be the set $\{0,1\}^{\mathbb{N}}$ equipped with the Borel $\sigma$-algebra   determined by the metric $d\big(\{x_n\}_{n\in \mathbb{N}},  \{y_n\}_{n\in \mathbb{N}} \big)=\sum_{n\in \mathbb{N}}2^{-n}|x_n-y_n|  $. 
  For a $\mu$-integrable function $\psi:\Gamma\rightarrow \R$, let $ \mathrm{E}_{\mu}\left[\psi\,|\,\mathcal{F}_n\right]$ denote the conditional expectation of $\psi$ with respect to  $\mathcal{F}_n$, and let  $P_n:L^2(\Gamma,\mu)\rightarrow L^2(\Gamma,\mu)$ denote the corresponding orthogonal projection.   The projections $P_n$ converge strongly as $n\rightarrow \infty$  to the identity operator on $L^2(\Gamma,\mu)$ since the algebra $\mathcal{A}:=\bigcup_{n=1}^{\infty}\mathcal{F}_n$ generates $\mathcal{B}_{\Gamma}$.

\begin{proof}[Proof of  Corollary~\ref{CorApprox}] Let $Y:\mathcal{H}\rightarrow L^2(\Gamma,\mu)$ be a bounded linear operator for which $\mathbf{T}=YY^*$ is  Hilbert-Schmidt  with kernel $T(p,q)$ satisfying $\int_{\Gamma\times \Gamma}\textup{exp}\big\{ T(p,q) \big\}\mu(dp)\mu(dq)<\infty$. 
To show that $Y$ is a randomized shift, it suffices to verify conditions (I) and (II) of Theorem~\ref{ThmShamovConv} for the sequence $Y_{n}:=P_n Y$.   Note that $Y_{n}$ is a finite-dimensional operator and thus a trivial randomized shift, i.e.,  $Y_{n}(p)\in\mathcal{H}$ for $\mu$-a.e.\ $p\in \Gamma$. Let   $\mathbf{M}_{Y_n}$ be the GMC  associated to $Y_{n} $ with  expectation $\mu$.    The second moment of the total mass of $\mathbf{M}_{Y_n}$ has the bound $$\mathbb{E}\Big[\big( \mathbf{M}_{Y_n}(\Gamma)  \big)^2\,\Big| \,W\Big]\,=\,\int_{\Gamma\times \Gamma}e^{ T_{n}(p,q)  }\mu(dp)\mu(dq) \,\leq \,\int_{\Gamma\times \Gamma}e^{ T(p,q)  }\mu(dp)\mu(dq)\,<\,\infty \,, $$
     where the first inequality holds by Jensen's inequality since $T_n(p,q)=\mathrm{E}_{\mu}\big[T(p,q)\,|\,\mathcal{F}_{n}\otimes  \mathcal{F}_{n}  \big]   $.  It follows that the family of random variables $\big\{\big( \mathbf{M}_{Y_n}(\Gamma)\big\}_{n\in \mathbb{N}}$ is uniformly integrable, which is condition (I).  To verify condition (II) note that the projections $P_n$ converge strongly to the identity operator on $L^2(\Gamma,\mu)$,  the operators $Y_{n}:=P_{n} Y$ converge strongly to $Y $ (and, in fact, in operator norm because $Y $ must be compact for $YY^*$ to be Hilbert-Schmidt).

\end{proof}

\begin{proof}[Proof of Proposition~\ref{PropGMCmoments}] 
Let $Y:\mathcal{H}\rightarrow L^2(\Gamma,\mu)$ be a randomized shift such that $\mathbf{T}=YY^*$ is Hilbert-Schmidt with   kernel $T(p,q)$ having finite exponential moments. Define $Y_n$, $T_n(p,q)$, and  $\mathbf{M}_{Y_n}$ as in the proof of  Corollary~\ref{CorApprox}.   The sequence $T_n(p,q)=\mathrm{E}_{\mu}\big[T(p,q)\,|\,\mathcal{F}_{n}\otimes  \mathcal{F}_{n}  \big]   $ forms a martingale with respect to the filtration 
$\big(\mathcal{F}_{n}\otimes  \mathcal{F}_{n}\big)_{n\in \mathbb{N}}$ that converges $\mu\times \mu$-a.e.\ to $T(p,q)$, and hence condition (III) of Theorem~\ref{ThmShamovConv} holds.  
Since conditions (I) and (II)  of Theorem~\ref{ThmShamovConv}  hold by the proof  of Corollary~\ref{CorApprox}, Theorem~\ref{ThmShamovConv} implies that $\mathbf{M}_{Y_n}(A)$ converges in probability to $\mathbf{M}_{Y}(A)$ for any measurable $A\subset \Gamma$.

 Since the exponential moments of $T(p,q)$ with respect to $\mu\times \mu$ are finite,  it follows that 
\begin{align}\label{IntFinite}
 \int_{A^m}\textup{exp}\Bigg\{\sum_{1\leq i<j\leq m }  T(p_i , p_j)  \Bigg\} \mu(dp_1)\cdots \mu(dp_m)
\end{align}
is finite for any  $m\in \mathbb{N}$.  The $m^{th}$ moment of $\mathbf{M}_{Y_n}(A)$ is the random variable
\begin{align}
\mathbb{E}\left[ \big( \mathbf{M}_{Y_n}(A)  \big)^m  \right]\,=\,&\mathbb{E}\Bigg[ \bigg( \int_{A}\textup{exp}\bigg\{\big\langle W, Y_n(p) \big\rangle -\frac{1}{2}\mathbb{E}\Big[\big\langle W, Y_n(p) \big\rangle^2\Big]   \bigg\}   \mu(dp) \bigg)^m  \Bigg]\nonumber \\
\,=\,&  \int_{A^m}\mathbb{E}\Bigg[\prod_{j=1}^m\textup{exp}\bigg\{\big\langle W, Y_n(p_j) \big\rangle -\frac{1}{2}\mathbb{E}\left[\big\langle W, Y_n(p_j) \big\rangle^2\right]   \bigg\} \Bigg]  \mu(dp_1)\cdots \mu(dp_m) \,.\nonumber 
\intertext{The random variables in the product are log-normal, so the above is equal to}
\,=\,&  \int_{A^m}\textup{exp}\Bigg\{\sum_{1\leq i<j\leq m } \mathbb{E}\big[\big\langle W, Y_n(p_i) \big\rangle\big\langle W, Y_n(p_j) \big\rangle \big] \Bigg\}  \mu(dp_1)\cdots \mu(dp_m)\,,\nonumber 
\intertext{and since $\mathbb{E}\big[\big\langle W, Y_n(p) \big\rangle\big\langle W, Y_n(q) \big\rangle \big]=\big\langle Y_n(p),\, Y_n(q) \big\rangle = T_n(p , q)  $, we get   }
\,=\,&  \int_{A^m}\textup{exp}\Bigg\{\sum_{1\leq i<j\leq m }  T_n(p_i , p_j)  \Bigg\} \mu(dp_1)\cdots \mu(dp_m)\,.\label{UH}
\end{align}
Since $T_n(p,q)$ is the conditional expectation of $T(p,q)$ with respect to $\mathcal{F}_{n}\otimes  \mathcal{F}_{n}  $,  Jensen's inequality implies that~(\ref{UH}) is bounded by~(\ref{IntFinite}).  Since $T_n(p,q)$ converges $\mu\times \mu$-a.e.\ to $T(p,q)$, the expression~(\ref{UH}) converges to~(\ref{IntFinite}) by Fatou's lemma.  Moreover,  since $\sup_{n\in \mathbb{N}}\mathbb{E}\left[ \big( \mathbf{M}_{Y_n}(A)  \big)^m  \right]  $ is finite for each $m$ and     the random variables $\mathbf{M}_{Y_n}(A)$ converge in probability  to $ \mathbf{M}(A)$, the moments $\mathbb{E}\left[ \big( \mathbf{M}_{Y_n}(A)  \big)^m  \right]$ converge to $\mathbb{E}\left[ \big( \mathbf{M}(A)  \big)^m  \right]$ as $n\rightarrow \infty$ for each $m\in \mathbb{N}$.  Hence $\mathbb{E}\left[ \big( \mathbf{M}(A)  \big)^m  \right]$ is equal to~(\ref{IntFinite}).

\end{proof}

\section{Construction of the isometry $\hat{U}_{\mathbf{M}_r}$}\label{SecA}

The construction of $Y_{\mathbf{M}_r}$ in Definition~\ref{DefY} requires a linear, isometric embedding $\hat{U}_{\mathbf{M}_r}:L^2(D,\vartheta_{\mathbf{M}_r}) \rightarrow  \mathcal{H}$ that is a measurable function of $\mathbf{M}_r$.  Given a realization of the measure $(D,\vartheta_{\mathbf{M}_r})$, we can construct such a map $\hat{U}_{\mathbf{M}_r}$ using, for instance, the basic recipe below. 
\begin{itemize}
\item  Let $C(D)$ be the space of real-valued continuous functions on $D$ equipped with the uniform norm: $\|f\|_{\infty} =\max_{x\in D}|f(x)|$. Let $\frak{D}$ be a countable, dense subset  of the unit shell $\big\{f\in  C(D)\,\big|\,\|f\|_{\infty} = 1   \big\}$ and  $\{h_n\}_{n\in \mathbb{N}}$ be an enumeration of $\frak{D}$.

\item  Apply Gram-Schmidt to $\{h_n\}_{n\in \mathbb{N}}$ to generate an orthonormal sequence  $\big\{\mathbf{h}_n^{(r)}\big\}_{n\in \mathbb{N}}$ under the inner product
\begin{align*}
\text{ }\hspace{1cm}\big\langle \phi_1, \phi_2 \big\rangle_{L^{\mathsmaller{2}}(D, \vartheta_{\mathbf{M}_r}) } \,=\,\int_{D}\phi_1(x)\phi_2(x)\vartheta_{\mathbf{M}_r}(dx)\,, \hspace{1.2cm} \phi_1, \phi_2 \in L^{2}(D, \vartheta_{\mathbf{M}_r})\,.
\end{align*}
Let $c_{k,n}^{(r)}\in \R$ denote  the Gram-Schmidt coefficients, i.e., the values satisfying   $h_n=\sum_{k=1}^n c_{k,n}^{(r)}\mathbf{h}_k^{(r)}$.

\item Let $\{e_n\}_{n\in \mathbb{N}}$ denote an orthonormal basis of $\mathcal{H}$.  Define an isometric map $\hat{U}_{\mathbf{M}_r}:\textup{Span}\,\big\{\mathbf{h}_n^{(r)}\big\}_{n\in \mathbb{N}} \rightarrow \mathcal{H} $ by sending $\mathbf{h}_n^{(r)}\mapsto e_n$ for all $n\in \mathbb{N}$.   The algebraic span of $\big\{\mathbf{h}_n^{(r)}\big\}_{n\in \mathbb{N}}$ must be dense in $L^{2}(D, \vartheta_{\mathbf{M}_r})$ since $\|\phi\|_{L^{2}(D, \vartheta_{\mathbf{M}_r})}  \leq\sqrt{ \vartheta_{\mathbf{M}_r}(D)    }\|\phi\|_{\infty}$ and $\vartheta_{\mathbf{M}_r}$  is a.s.\ finite.  Therefore $\hat{U}_{\mathbf{M}_r}$ extends to an isometry from $L^{2}(D, \vartheta_{\mathbf{M}_r})$ into $\mathcal{H}$.

\end{itemize}
 The linear map $\hat{U}_{\mathbf{M}_r}$ depends measurably on $\mathbf{M}_r$ since $\vartheta_{\mathbf{M}_{r}}$ depends measurably on $\mathbf{M}_{r}$.

\end{appendix}

\end{document}